\documentclass[11pt,reqno]{article}
\usepackage[a4paper,margin=1.1in]{geometry}
\usepackage{microtype}
\usepackage{psfrag}
\usepackage[parfill]{parskip}
\usepackage{float}
\usepackage{graphicx}
\usepackage{epsfig}
\usepackage{amsmath, amsthm, amssymb}
\usepackage{verbatim}
\usepackage{multicol}
\usepackage{url}
\usepackage{latexsym}
\usepackage{mathtools}
\usepackage{mathrsfs}
\usepackage[colorlinks = true,
            linkcolor  = blue,
            citecolor  =blue,
            urlcolor   =blue,
            bookmarks  = true]{hyperref}
 \usepackage{enumitem}
\setlist[enumerate,1]{label=(\roman*)}
\usepackage[normalem]{ulem}
\usepackage{bm}
\usepackage{dsfont}
\usepackage{stmaryrd}
\usepackage[dvipsnames]{xcolor}

\usepackage[backend=biber, style=numeric, citestyle=numeric-comp, giveninits=true, doi=false, isbn=false, url=false, eprint=false, maxbibnames=99]{biblatex}
\DeclareFieldFormat{labelnumberwidth}{\mkbibbrackets{#1}}
\DeclareNameAlias{author}{given-family}
\DeclareNameAlias{editor}{given-family}

\DeclareFieldFormat[article,inproceedings,incollection,misc,unpublished,book]{title}{#1}
\DeclareFieldFormat[article]{journaltitle}{#1}
\DeclareFieldFormat[article]{volume}{#1}

\AtEveryBibitem{%
  \clearfield{issn}
  \clearfield{url}
  \clearfield{eprint}
  \clearfield{doi}
  \clearfield{isbn}
  \clearfield{number}
  \clearfield{mrnumber}
  \clearfield{mrclass}
  \clearfield{reviewer}
}

\renewbibmacro{in:}{}

\renewbibmacro*{volume+number+eid}{%
  \printfield{volume}%
  \setunit{\addcomma\space}%
  \printfield{eid}}

\DeclareDelimFormat{namedelim}{\addcomma\space}

\DeclareFieldFormat{pages}{#1}

\numberwithin{equation}{section}

\theoremstyle{plain}

\newtheorem{theorem}{Theorem}[section]
\newtheorem{lemma}[theorem]{Lemma}
\newtheorem{proposition}[theorem]{Proposition}
\newtheorem{corollary}[theorem]{Corollary}

\theoremstyle{definition}

\newtheorem{definition}[theorem]{Definition}

\newtheorem{remark}[theorem]{Remark}

\DeclarePairedDelimiterX{\ip}[2]{\langle}{\rangle}{#1, #2}
\DeclarePairedDelimiterX{\norm}[1]{\lVert}{\rVert}{#1}

\newcommand{\bS}{\mathbb{S}}
\newcommand{\R}{\mathbb{R}}

\usepackage{cleveref}
  \crefname{theorem}{Theorem}{Theorems}
  \crefname{thm}{Theorem}{Theorems}
  \crefname{thm*}{Theorem*}{Theorems}
  \crefname{lemma}{Lemma}{Lemmas}
  \crefname{lem}{Lemma}{Lemmas}
  \crefname{remark}{Remark}{Remarks}
  \crefname{prop}{Proposition}{Propositions}
\crefname{notation}{Notation}{Notations}
\crefname{claim}{Claim}{Claims}
  \crefname{defn}{Definition}{Definitions}
  \crefname{corollary}{Corollary}{Corollaries}
  \crefname{section}{Section}{Sections}
  \crefname{figure}{Figure}{Figures}
    \crefname{assumption}{Assumption}{Assumptions}

\usepackage{orcidlink}

\begin{document}

\title{The smallest singular value of products of random combinatorial and deterministic matrices}

\author{%
  Yuxiao Shi, Kexin Yu, and Tingzhou Yu%
  \raisebox{0.6ex}{\scriptsize\orcidlink{0009-0005-0414-7522}}%
}
\date{}

\maketitle

\begin{abstract}
Let $p\in(0,1/2]$ be fixed, and let $d,n$ be positive integers satisfying $d=pn$. Let $Q_n$ be an $n\times n$ random combinatorial matrix whose rows are independent and uniformly distributed over the vectors in $\{0,1\}^n$ having exactly $d$ ones. For an arbitrary deterministic invertible $n\times n$ matrix $M$, we establish two-sided bounds for the smallest singular value of $Q_nM$. More precisely, for every $\varepsilon\in(0,1)$,
\[
\mathbb P\!\left(
\varepsilon\le s_n(Q_nM)\cdot \|M^{-1}\|_{\mathrm{HS}}
\le \varepsilon^{-2}
\right)
\ge 1-C\left(
\varepsilon+d^{-1/2}+e^{-c\,\operatorname{sr}(M^{-1})}
\right)-2e^{-cn},
\]
where positive constants $C, c$ depend only on $p$ and $\operatorname{sr}$ denotes the stable rank.
Consequently, for any sequence of deterministic invertible matrices $M$ satisfying $\operatorname{sr}(M^{-1})\to\infty$, $s_n(Q_nM)$ is typically of order $\|M^{-1}\|_{\mathrm{HS}}^{-1}$. In particular, taking $M=\mathbf I_n$ recovers the known $n^{-1/2}$ scale for $s_n(Q_n)$.
\end{abstract}


\section{Introduction}\label{sec1}
Let $A$ be an $n \times n$ real matrix, and let $s_n(A)$ denote its smallest singular value. It admits the following variational characterization:
\[
s_n(A)=\inf_{x \in \mathbb{S}^{n-1}} \|Ax\|_2.
\]
 The quantity $s_n(A)$ measures the distance from $A$ to the set of singular matrices in the operator norm. In particular, $s_n(A)>0$ if and only if $A$ is invertible, in which case $s_n(A)={\|A^{-1}\|}^{-1}$.

 The quantitative study of the smallest singular value goes back to von Neumann and Goldstine's work \cite{numericalinverting,collectedworks}. For matrices with i.i.d.\ Gaussian entries, Edelman and Szarek established the $n^{-1/2}$ scale for the smallest singular value \cite{eigenvalues,conditionnumbers}. Subsequent work obtained quantitative estimates for a wider class of independent entry distributions \cite{rudelson2008invertibility,tao2009inverse, thelittle, rudelson2009smallest,tao2010random, coveringof,livshyts2021smallestTK}, including rectangular \cite{litvak2005smallest, fernandez2025logconcave, dabagia2024inhomogeneous, fernandez2024distance} and sparse \cite{litvak2022singularity, yu2025sparsediscrete, basakrudelson2017invertibility, basak2021sharp} matrices. See \cite{tikhomirov2022quantitative} for a survey. 
For an $n\times n$ matrix $B$ with i.i.d. entries distributed as a real subgaussian random variable $\xi$ with mean zero and variance one, a recent result by Sah, Sahasrabudhe, and Sawhney \cite[Theorem 1.1]{sah2025spielmanteng} sharpens the earlier result \cite{thelittle} to
\[
\mathbb P\!\left(s_n(B)\le\frac{\varepsilon}{\sqrt n}\right)
\le (1+o_n(1))\varepsilon+e^{-c_\xi n},
\qquad \varepsilon\ge 0,
\]
where $c_\xi>0$ depends only on $\xi$.

Progress has been slower for matrices with dependent entries, where many approaches for the independent entry case no longer apply. For example, the bound of the smallest singular value for symmetric matrices with i.i.d.\ mean zero, variance one subgaussian entries on and above the diagonal was obtained only recently \cite{symmetricsubgaussian}. Other examples come from adjacency matrices of random regular directed graphs \cite{cookregular,litvakdigraphs,litvakshiftedregular,jainregular}. 

In this paper, we study the smallest singular value of the product $Q_nM$, where $Q_n$ is the $n\times n$ random matrix  whose rows are independent and uniformly distributed over the set of vectors $\{0,1\}^n$ having exactly $d=pn$ ones and $M$ is an arbitrary deterministic invertible $n\times n$ real matrix. Such a matrix is often called a random combinatorial matrix in the literature. The matrix $Q_n$ may also be viewed as the adjacency matrix of a random directed graph in which every vertex has a fixed out-degree $d$, with self-loops allowed.

When $M=\mathbf I_n$, the singularity problem for $Q_n$ was initiated by Nguyen \cite{nguyencombinatorial}, with further results for the smallest singular value \cite{jainapprox,tran} and singularity probability in \cite{ferberjainluhsamotij}. In the dense regime $d=pn$ with fixed $p\in (0,1/2]$, Tran's lower bound \cite[Theorem 1.2 and Remark (3)]{tran} and the complementary upper bound of D. Li, A. Litvak, and T. Yu \cite{lianupper} determine the typical $n^{-1/2}$ scale of $s_n(Q_n)$. In the sparse regime $C\log n \le d \le n/2$, quantitative lower bounds for the smallest singular value of shifts of $Q_n$ were obtained in \cite{lianlitvakyucircular}. For a related signed random combinatorial matrix, the lower bound for the smallest singular value in the dense regime was proved by K. Yu in \cite{yusigned}. A complementary upper bound for the same signed model was obtained in \cite{shi2026signedupper}. None of these results covers the product $Q_nM$ with a general deterministic invertible matrix $M$.

For products of random and deterministic matrices, earlier work includes spectral norm estimates \cite{vershynin2011products} and a local circular law \cite{xiyangyin2017product}. Products of several independent random matrices have also been studied through least singular value estimates and universality \cite{kopel2020products,chaudhuri2021products}, and through their $\ell^\infty\!\to\!\ell^\infty$ operator norm \cite{mourrat2026products}.
Recently, for a matrix $A$ with independent, mean zero, variance one entries satisfying uniform fourth moment and anti-concentration bounds, and a deterministic invertible matrix $M$, Letwin and Polavarapu \cite{letwin2026smallest} proved that
\[
\mathbb P\left(s_n(MA)\le \frac{\varepsilon}{\|M^{-1}\|_{\mathrm{HS}}}\right)
\le C\varepsilon+e^{-cn}.
\]
For an $n\times n$ Gaussian matrix $G_n$,  they also proved that $\mathbb E\,s_n(MG_n)\asymp\|M^{-1}\|_{\mathrm{HS}}^{-1}$ for every deterministic invertible $M$, and that $s_n(MG_n)\le Ct\|M^{-1}\|_{\mathrm{HS}}^{-1}$ with probability at least $1-Ce^{-ct^2}-Ce^{-c\operatorname{sr}(M^{-1})}$ in \cite[Theorem 5.1]{letwin2026smallest},  where $\operatorname{sr}(A)=\|A\|_{\mathrm{HS}}^2/\|A\|^2$ denotes the stable rank of a matrix $A$.

Motivated by this result, we consider the corresponding problem for random combinatorial matrices. The matrix $M$ may distort different directions by different amounts, so the typical scale $n^{-1/2}$ for the smallest singular value is no longer expected. For every deterministic invertible $M$, we prove a lower bound at the scale $\|M^{-1}\|_{\mathrm{HS}}^{-1}$, and we also prove that this is the typical scale whenever $\operatorname{sr}(M^{-1})\to \infty$.

Our first main theorem establishes a lower bound for the smallest singular value of $Q_nM$.
\begin{theorem}\label{theorem 1.1}
Let $p\in (0,1/2]$ be fixed. Let $d, n\in \mathbb{N}$ be such that $d=pn$. Let $Q_n$ be an $n\times n$ random matrix whose rows are independent and uniformly distributed over the set of vectors in $\{0,1\}^{n}$ with exactly $d$ ones. Then there exist constants  $C_{\ref{theorem 1.1}},c_{\ref{theorem 1.1}}$, depending only on $p$, such that for every deterministic invertible $n\times n$ matrix $M$ and every $\varepsilon\ge 0$,
\begin{equation}
\mathbb{P}\left( s_n(Q_nM) \le \frac{\varepsilon}{\|M^{-1}\|_{\mathrm{HS}}} \right)
\le  C_{\ref{theorem 1.1}} \varepsilon + 2e^{-c_{\ref{theorem 1.1}}n}.
\nonumber
\end{equation}
\end{theorem}

\begin{remark}\label{rem:lowerbound}
Theorem~\ref{theorem 1.1} improves the bound obtained by combining the inequality $s_n(Q_nM)\ge s_n(Q_n)s_n(M)$ with the known lower bound for $s_n(Q_n)$ \cite{tran},
\[
\mathbb P\!\left(s_n(Q_nM)\le\varepsilon\frac{s_n(M)}{\sqrt n}\right)
\le C\varepsilon+2e^{-cn},\qquad \varepsilon\ge0.
\]
Since \[
\frac{1}{\|M^{-1}\|_{\mathrm{HS}}}\ge \frac{1}{\sqrt n\,\|M^{-1}\|
}=\frac{s_n(M)}{\sqrt n},
\]
the lower bound in Theorem \ref{theorem 1.1} is larger than the above lower bound.
\end{remark}

The following theorem provides a complementary upper bound for $s_n(Q_nM)$.
\begin{theorem}\label{theorem 1.2}
Under the same assumptions as in Theorem \ref{theorem 1.1}, there exist constants $C_{\ref{theorem 1.2}},c_{\ref{theorem 1.2}}>0$, depending only on $p$, such that for every deterministic invertible $n\times n$ matrix $M$ and every $\varepsilon> 0$,
\begin{equation}
\mathbb{P}\left(s_n(Q_nM)>\frac{1}{\varepsilon^2\|M^{-1}\|_{\mathrm{HS}}}\right)\le C_{\ref{theorem 1.2}} \left(\varepsilon +\frac{1}{\sqrt{d}}+e^{-c_{\ref{theorem 1.2}} \operatorname{sr}(M^{-1})}\right),
\nonumber
\end{equation}
where $\operatorname{sr}(M^{-1})=\|M^{-1}\|_{\mathrm{HS}}^2/\|M^{-1}\|^2$ denotes the stable rank of $M^{-1}$.
\end{theorem}

\begin{remark}
Combining Theorem~\ref{theorem 1.1} and Theorem~\ref{theorem 1.2}, we show that $s_n(Q_nM)$ is typically of order $\|M^{-1}\|_{\mathrm{HS}}^{-1}$ when $\operatorname{sr}(M^{-1})\to\infty$. More precisely, for every $\eta>0$, there are constants $0<a_\eta<b_\eta$, depending only on $p$ and $\eta$, such that, for all sufficiently large $n$,
\[
\mathbb P\!\left(a_\eta
\le s_n(Q_nM) \cdot \|M^{-1}\|_{\mathrm{HS}}\le b_\eta\right)\ge1-\eta.
\]
We provide two examples as follows.
\begin{enumerate}
\item If $M=\mathbf I_n$, then $\|M^{-1}\|_{\mathrm{HS}}=\sqrt n$ and $\operatorname{sr}(M^{-1})=n$. For each $\eta>0$ and all sufficiently large $n$,
\[
\mathbb P\!\left(a_\eta \le s_n(Q_n)\cdot n^{1/2}\le b_\eta \right)\ge1-\eta.
\]
\item Let $k_n=\lfloor\sqrt n\rfloor$ and $M=\operatorname{diag}(n^{-1}\mathbf I_{k_n},\mathbf I_{n-k_n})$. In this case,
\[
\begin{aligned}
L_n:=\|M^{-1}\|_{\mathrm{HS}}^{-1}=(k_n n^2+n-k_n)^{-1/2}\sim n^{-5/4}, \qquad 
\operatorname{sr}(M^{-1})=k_n+\frac{n-k_n}{n^2}\sim\sqrt n.
\end{aligned}
\]
For $0<\varepsilon<1$, by Theorem~\ref{theorem 1.1} and Theorem~\ref{theorem 1.2},
\[
\begin{aligned}
\mathbb P\!\left(\varepsilon <s_n(Q_nM)\cdot L_n^{-1}\le\varepsilon^{-2}\right)\ge1-C\left(\varepsilon+n^{-1/2}+e^{-c\sqrt n}\right)-2e^{-cn},
\end{aligned}
\]
where $C,c>0$ depend only on $p$. Hence $s_n(Q_nM)$ is typically of order $n^{-5/4}$. The elementary product argument in Remark~\ref{rem:lowerbound} gives only a lower bound of order $n^{-3/2}$.
\end{enumerate}

The stable rank can be estimated from the singular values of $M$. Writing $s_1(M)\ge\cdots\ge s_n(M)>0$, we have
\[
\operatorname{sr}(M^{-1})
=\sum_{j=1}^n\left(\frac{s_n(M)}{s_j(M)}\right)^2
\ge\frac{n}{\kappa(M)^2},
\qquad
\kappa(M):=\frac{s_1(M)}{s_n(M)}.
\]
Thus, $\kappa(M)=o(\sqrt n)$ is sufficient for $\operatorname{sr}(M^{-1})\to\infty$. The condition does not always hold. For example, if $M=\operatorname{diag}(n^{-1},1,\ldots,1)$, then $\|M^{-1}\|_{\mathrm{HS}}^{-1}\sim n^{-1}$ but $\operatorname{sr}(M^{-1})\to 1$. Theorem \ref{theorem 1.1} still improves the lower bound from order $n^{-3/2}$ to order $n^{-1}$, while Theorem \ref{theorem 1.2} does not establish a matching typical order for this example. For Gaussian matrices, the stable rank term in the probability bound is in general unavoidable \cite[Proposition 5.5]{letwin2026smallest}, although $\mathbb E\,s_n(MG_n)\le C\|M^{-1}\|_{\mathrm{HS}}^{-1}$  for every deterministic invertible matrix $M$ \cite[Proposition 5.3]{letwin2026smallest}.
\end{remark}

\subsection{An application: quantitative stability of random participation matrices}
Let $d=pn$ be as in Theorem \ref{theorem 1.1}, and let $x_1,\ldots,x_n\in\mathbb R^m$ be fixed unknown vectors, where $m\in\mathbb N$. For each $t\in[n]$, let $S_t$ be sampled independently and uniformly from the family of all $d$-element subsets of $[n]$. Thus, $S_t$ is the set of indices selected at round $t$, and its indicator vector $(\mathbf1_{\{i\in S_t\}})_{i=1}^n\in\{0,1\}^n$ has exactly $d$ ones. Given $S_t$, observe the weighted sum
\[
y_t=\sum_{i\in S_t}w_i x_i+z_t,
\]
where the weights $w_i>0$ are deterministic and known, and $z_t\in\mathbb R^m$ is a perturbation. Define the participation matrix $P_n\in\{0,1\}^{n\times n}$ by
\[
(P_n)_{ti}:=\mathbf1_{\{i\in S_t\}},\qquad t, \, i\in[n],
\]
Let $D_n:=\operatorname{diag}(w_1,\ldots,w_n)$. Let $X$, $Y$, and $Z$ have rows $x_i$, $y_t$, and $z_t$, respectively. Then $P_n$ has the same distribution as $Q_n$, and the observations satisfy
\[
Y=P_nD_nX+Z,
\]
with $X,Y,Z\in\mathbb R^{n\times m}$. 

Participation matrices of this type arise in federated learning under fixed size client sampling and weighted aggregation \cite{li2020fedavg,fraboni2021clustered}. Related multi-round reconstruction problems have been studied in \cite{lam2021gradient,so2023securing}. Our result quantifies the worst-case stability of the corresponding dense square model.

Once $P_n$ and $D_n$ are known, define $\widehat X=(P_nD_n)^{-1}Y$ whenever $P_nD_n$ is invertible. We define the worst-case noise amplification factor, equivalently the absolute condition number of the solution map with respect to perturbations in the observations, by
\[
\mathcal{A}_n
:=\sup_{Z\in\mathbb R^{n\times m}\setminus\{0\}}
\frac{\|\widehat X-X\|_{\mathrm{HS}}}{\|Z\|_{\mathrm{HS}}}
=\|(P_nD_n)^{-1}\|
=\frac{1}{s_n(P_nD_n)},
\]
with the convention that $\mathcal{A}_n=\infty$ when $P_nD_n$ is singular. Here $P_n$, $D_n$, and $X$ are fixed and the supremum is over the noise $Z$. So $\mathcal{A}_n$ is the worst case noise amplification factor (see e.g., \cite{arioli2007partial,baboulin2009conditioning})
\[
\|\widehat X-X\|_{\mathrm{HS}}
\le \mathcal A_n\|Z\|_{\mathrm{HS}}.
\]
Let
\[
H_n:=\left(\sum_{i=1}^n w_i^{-2}\right)^{1/2}
\qquad\text{and}\qquad
r_n:=\frac{\sum_{i=1}^n w_i^{-2}}
{\max_{i\in[n]}w_i^{-2}}.
\]
\begin{corollary}\label{corollary participation stability}
In the model above, assume that all weights $w_i$ are deterministic and positive. There are constants $C,c>0$, depending only on $p$, such that, for every $0<\varepsilon\le1$,
\[
\mathbb{P}\left(\mathcal{A}_n\ge\frac{H_n}{\varepsilon}\right)
\le C\varepsilon+2e^{-cn}
\]
and
\[
\mathbb{P}\left(\mathcal{A}_n<\varepsilon^2H_n\right)
\le C\left(\varepsilon+\frac{1}{\sqrt d}+e^{-cr_n}\right).
\]
\end{corollary}

\begin{proof}
 Note that
\[
\|D_n^{-1}\|_{\mathrm{HS}}=H_n, \qquad
\operatorname{sr}(D_n^{-1})=r_n.
\]
The two probability estimates therefore follow from Theorems \ref{theorem 1.1} and \ref{theorem 1.2} with $M_n=D_n$
\end{proof}

\begin{remark}
If $r_n\to\infty$, Corollary~\ref{corollary participation stability} yields
\[
\mathcal{A}_n=\Theta_{\mathbb P}(H_n).
\]
In particular, for the unweighted model $D_n=\mathbf I_n$,
\[
\mathcal{A}_n=\Theta_{\mathbb P}(\sqrt n).
\]
Thus, while invertibility of $P_nD_n$ guarantees noiseless identifiability, its least singular value controls the worst-case amplification of perturbations. Under the condition above, this amplification is typically of order $H_n$.
\end{remark}

\subsection{Outline of the proof }
The proof of Theorem \ref{theorem 1.1} follows the invertibility-via-distance approach of Rudelson and Vershynin \cite{thelittle} and the strategy of Letwin and Polavarapu \cite[Theorem 4.4]{letwin2026smallest}. 

The proof of Theorem \ref{theorem 1.2} is based on a random permutation argument. By introducing an independent uniform permutation matrix $\Pi$, we reduce the problem to estimating the lower tail of $\|M^{-1}\Pi v\|_2$ for non-almost-constant vectors $v$. Since the coordinates of $\Pi v$ are dependent, the Hanson-Wright inequality cannot be applied directly. To overcome this difficulty, we construct a weighted matching and introduce independent Rademacher swaps along the matched pairs. The matching captures a substantial portion of the relevant Hilbert-Schmidt norm, while the Rademacher swaps preserve the uniform permutation law and introduce sufficient independent randomness for an application of the Hanson-Wright inequality. This yields a lower bound for the permuted norm, with the stable-rank term arising naturally from the resulting bound. Combining this estimate with the known upper bound for $s_n(Q_n)$ in \cite{lianupper}, this completes the proof.
\subsection{Organization}
The remainder of the paper is organized as follows. Section 2 collects the notation and preliminary results. In Section 3, we prove the lower bound in Theorem \ref{theorem 1.1} using the invertibility-via-distance approach and the average projection estimate. In Section 4, we develop the random permutation argument and prove the upper bound in Theorem \ref{theorem 1.2}.

\section{Preliminaries}\label{sec2}

\subsection{Notation}
 We write $[n]:=\{1,2,\dots, n\}$ and denote by $S_n$ the symmetric group on $[n]$, that is, the set of all permutations of $[n]$. For $\pi\in S_n$, let $\Pi$ denote the corresponding permutation matrix, which is defined by 
\[
\Pi e_i=e_{\pi(i)},\quad i\in[n],
\]where $\{e_i\}_{
i=1}^n$ is the standard basis of $\R^n$. The standard inner product and Euclidean norm are denoted by $\langle\cdot,\cdot\rangle$ and $\norm{\cdot}_2$, respectively.
The vector $\mathbf 1_n\in \mathbb R^n$ denotes the vector with each coordinate equal to $1$, and $\mathbf{I}_n$ denotes the $n\times n$ identity matrix. We write \[\mathbf 1_n^\perp:=\{v\in \mathbb{R}^n: \langle v, \mathbf 1_n\rangle=0\}, \quad P_{\mathbf 1_n^\perp}:=\mathbf I_n-\frac1n \mathbf 1_n\mathbf 1_n^T, \quad P_{\mathbf{1}_n}:=\frac{1}{n}\mathbf 1_n\mathbf 1_n^T.\] 
Here $P_{\mathbf 1_n^\perp}$ and $P_{\mathbf{1}_n}$ are the orthogonal projections onto $\mathbf 1_n^\perp$ and $\mathrm{span}\{\mathbf 1_n\}$, respectively.
We denote the unit sphere in $\R^n$ by $\bS^{n-1}:=\{x\in \mathbb R^n: \|x\|_2=1\}$ and the Euclidean unit ball by $B_2^n$. The distance from a vector $v$ to a subspace $H\subset \mathbb R^n$ is defined as \[\operatorname{dist}(v,H):=\inf_{u\in H}\norm{v-u}_2.\]
For a real matrix $A=(a_{ij})\in\mathbb R^{m\times k}$, its operator and Hilbert--Schmidt norms are
\[
\|A\|:=\sup_{x\in\mathbb S^{k-1}}\|Ax\|_2,
\qquad
\|A\|_{\mathrm{HS}}:=\left(\sum_{i=1}^{m}\sum_{j=1}^{k}a_{ij}^2\right)^{1/2}.
\]
For a nonzero subspace $H\subset\mathbb R^k$, we write $A|_H$ for the restriction of $A$ to $H$ and define
\[
\|A|_H\|:=\sup_{\substack{x\in H\\ \|x\|_2=1}}\|Ax\|_2.
\]
In particular, when $k=n$,
\[
\|A|_{\mathbf 1_n^\perp}\|=\|AP_{\mathbf 1_n^\perp}\|.
\]
For $A\ne0$, its stable rank is $\operatorname{sr}(A):=\|A\|_{\mathrm{HS}}^2/\|A\|^2$. The transpose and inverse of $A$ are denoted by $A^T$ and $A^{-1}$, whenever the latter exists. For symmetric matrices $A,B$, we write $B\preceq A$ if $A-B$ is positive semidefinite, and $B\prec A$ if it is positive definite.

For a deterministic real sequence $a_n$ and a positive deterministic sequence $b_n$, we write $a_n=O(b_n)$ if $|a_n|\le Cb_n$ for some constant $C>0$ and all sufficiently large $n$, and $a_n=o(b_n)$ if $|a_n|/b_n\to0$. We write $a_n\sim b_n$ if $a_n/b_n\to1$. The subscript in $o_n(1)$ emphasizes that the limit is taken as $n\to\infty$. For positive deterministic sequences $a_n,b_n$, the notation $a_n\asymp b_n$ means that their ratio is bounded above and below by positive constants. For nonnegative random variables $X_n$, we write $X_n=\Theta_{\mathbb P}(a_n)$ if, for every $\eta>0$, there are $0<c_\eta<C_\eta<\infty$ such that
\[
\mathbb P(c_\eta a_n\le X_n\le C_\eta a_n)\ge1-\eta
\]
for all sufficiently large $n$. 
Finally, $C_0,c_0,C_1,c_1,\dots$ denote positive constants, possibly depending on the parameter $p$, whose values may change from line to line.

\subsection{Decomposition of the unit sphere}
We follow the compressible/incompressible decomposition \cite{thelittle}. Compressible vectors are close to sparse vectors, whereas incompressible vectors are necessarily spread over a substantial number of coordinates. Related decompositions appeared in earlier work on the smallest singular value of random matrices in \cite{litvak2005smallest} (see also \cite{litvak2004random}).

Fix $\delta,\rho\in(0,1)$ whose values will be chosen later and define the sets of sparse, compressible and incompressible vectors as follows:
\begin{align*}
    \operatorname{Sparse}(\delta)&:=\{v \in \mathbb{S}^{n-1}:|\operatorname{supp}(v)|\le\delta n\},\\
    \operatorname{Comp}(\delta,\rho)&:=\{v\in \mathbb{S
}^{n-1}:\operatorname{dist}(v,\operatorname{Sparse(\delta)})\le \rho\},\\
\operatorname{Incomp}(\delta,\rho)&:=\mathbb{S}^{n-1}\setminus\operatorname{Comp}(\delta,\rho).
\end{align*}
Another decomposition used in this paper splits the unit sphere into almost-constant and non-almost-constant vectors. This approach is particularly well suited to models with combinatorial dependence. It originates in Cook's work on random regular digraphs \cite{cookregular} and has been further developed and employed in \cite{tran,jaindiscrete,lianupper,yusigned}.

\begin{definition}
Let $\delta, \rho \in (0,1)$. We define $\operatorname{Cons}(\delta, \rho)$ as the set of vectors $v \in \bS^{n-1}$ such that $|v_i - \lambda| \le \rho/\sqrt{n}$ holds for some $\lambda \in \R$ and for at least $(1-\delta)n$ coordinates $i \in [n]$. A vector $v \in \bS^{n-1}$ is called non-almost-constant if  $v \in \bS^{n-1}\setminus \operatorname{Cons}(\delta, \rho)$.
\end{definition}
The following lemma from \cite[Lemma 2.2]{tran} shows that every non-almost-constant vector contains two large sets of coordinates that are uniformly separated.
\begin{lemma}\label{nonconstanseparated}
Let $\delta,\rho\in(0,1)$. Then for any vector
$v\in \mathbb{S}^{n-1}\setminus \operatorname{Cons}(\delta,\rho)$,
there are disjoint subsets
$\sigma_1=\sigma_1(v)$ and $\sigma_2=\sigma_2(v)$ of $[n]$ such that
\[
|\sigma_1|,|\sigma_2|\ge \frac{\delta n}{8},
\qquad \text{and} \qquad
\frac{\rho}{\sqrt{2n}}
\le |v_i-v_j|
\le \frac{6}{\sqrt{\delta n}}
\quad
\forall i\in \sigma_1,\ \forall j\in \sigma_2.
\]
\end{lemma}
\subsection{Invertibility over compressible and almost-constant vectors}
For $p=1/2$ Tran \cite[Proposition 2.22]{tran} established a probability estimate on the set of compressible vectors. Since the argument remains valid for every fixed $p\in(0,1/2]$, the same proof yields the following estimate.

\begin{lemma}\label{invertibility over compressible}
There exist constants $\delta_{\ref{invertibility over compressible}},\rho_{\ref{invertibility over compressible}}\in(0,1),c_{\ref{invertibility over compressible}},C_{\ref{invertibility over compressible}}$ depending only on $p$ such that
\[
\mathbb{P}\left(\inf_{x\in\operatorname{Comp}(\delta_{\ref{invertibility over compressible}},\rho_{\ref{invertibility over compressible}})}\|Q_n^{T}x\|_2\le C_{\ref{invertibility over compressible}}\sqrt{n}\right)\le 2 e^{-c_{\ref{invertibility over compressible}}n}.
\]
\end{lemma}
By applying Lemma \ref{invertibility over compressible}, we obtain the following desired probability estimate.
\begin{proposition}\label{compressible estimate}
    There exist constants $\delta_{\ref{compressible estimate}},\rho_{\ref{compressible estimate}}\in(0,1), C_{\ref{compressible estimate}}, c_{\ref{compressible estimate}}$ depending only on $p$ such that
    \[
    \mathbb{P}\left(\inf_{x\in\operatorname{Comp}(\delta_{\ref{compressible estimate}},\rho_{\ref{compressible estimate}})}\|M^TQ_n^Tx\|_2 \le \frac{C_{\ref{compressible estimate}} \sqrt{n}}{\|M^{-1}\|_{\mathrm{HS}}} \right)\le 2e^{-c_{\ref{compressible estimate}}n}.
    \]
\end{proposition}
\begin{proof}
We choose $\delta_{\ref{compressible estimate}},\rho_{\ref{compressible estimate}},C_{\ref{compressible estimate}},c_{\ref{compressible estimate}}$ to coincide, respectively, with those in Lemma \ref{invertibility over compressible}. Since 
\[
\|M^T Q_n^Tx\|_2\ge s_n(M^T)\|Q_n^Tx\|_2=\frac{\|Q_n^Tx\|_2}{\|(M^{T})^{-1}\|}\ge \frac{\|Q_n^Tx\|_2}{\|M^{-1}\|_{\mathrm{HS}}},
\]
we have 
\[
\left\{\inf_{x\in\operatorname{Comp}(\delta_{\ref{compressible estimate}},\rho_{\ref{compressible estimate}})}\|M^TQ_n^Tx\|_2 \le \frac{C_{\ref{compressible estimate}} \sqrt{n}}{\|M^{-1}\|_{\mathrm{HS}}} \right\}\subseteq \left\{\inf_{x\in\operatorname{Comp}(\delta_{\ref{invertibility over compressible}},\rho_{\ref{invertibility over compressible}})}\|Q_n^{T}x\|_2\le C_{\ref{invertibility over compressible}}\sqrt{n}\right\}.
\]
By using Lemma \ref{invertibility over compressible}, we can complete the proof.
\end{proof}
The following proposition provides an invertibility estimate for $Q_n$ on the set of almost-constant vectors (see \cite[Proposition 3.1]{lianupper}). 
\begin{proposition}\label{constant estimate}
Let $p\in(0,1/2]$ be a fixed constant. Let $d,n$ be large enough positive integers such that $d=pn$. There exist constants $C_{\ref{constant estimate}},c_{\ref{constant estimate}},\delta_{\ref{constant estimate}}, \rho_{\ref{constant estimate}}\in(0,1)$ depending only on $p$ such that the following holds. Let $M$ be a $m \times n$ random matrix, $ n/2\le m\le n$, whose rows are independent random vectors drawn uniformly from the set of $n$-dimensional $0/1$ vectors having exactly $d$ ones. Then
\[
\mathbb P\left(\inf_{v\in \operatorname{Cons}(\delta_{\ref{constant estimate}},\rho_{\ref{constant estimate}})}\|{M v}\|_2\le C_{\ref{constant estimate}}\sqrt{pn}\right)\le e^{-c_{\ref{constant estimate}}n}.
\]
\end{proposition}

\subsection{Anti-concentration}
For a random variable $X$ taking values in $\mathbb{R}$, the Lévy concentration function is defined as 
\[
\mathcal{L}(X,t):=\sup_{u\in \mathbb{R}}\mathbb{P}(|X-u|\le t),\qquad t\ge0.
\]
The concentration function measures the maximal probability that a random variable falls into a short interval. To estimate it under combinatorial dependence, Tran \cite{tran} introduced the Combinatorial Least Common Denominator (CLCD) for the unsigned $0/1$ model, and K.Yu \cite{yusigned} subsequently developed the corresponding signed estimates for the $\{-1,0,1\}$ model. We shall use these results below.

\begin{definition}\label{definition CLCD}
For a vector $v\in \R^n$, and parameters $\alpha,\gamma>0$, define 
\[
\operatorname{CLCD}_{\alpha,\gamma}(v):=\inf\left\{\theta>0:\operatorname{dist}\left(\theta \mathrm{D}(v),\mathbb{Z}^{\binom{n}{2}}\right)<\min\left(\gamma\norm{\theta\mathrm{D}(v)}_2,\alpha\right)\right\}.
\]
Here by $\mathrm{D}(v)$ we denote the vector in $\R^{\binom{n}{2}}$ whose $(i,j)$-coordinate is $v_i-v_j$, for $1\le i<j\le n$.
\end{definition}
The CLCD controls the concentration function through the following anti-concentration estimate. In the dense regime $d=pn$, the result follows directly from \cite[Lemma 2.9]{lianupper}. For convenience, we absorb all factors depending only on $p,b$, and $\gamma$ into a constant $C$. 
\begin{lemma}\label{anticoncentration estimate}
Fix $p\in(0,1/2],b>0$, and $\gamma \in(0,1)$. There exists a constant
$C_{\ref{anticoncentration estimate}}>0$,
depending only on $p,b,\gamma$, such that the following
holds. Let $n,d\in\mathbb{N}$ satisfy $d=pn$, and let $v\in \R^n $ satisfy $\norm{\mathrm{D}(v)}_2\ge b\sqrt{n}$. Suppose that  $\eta$ is a random vector uniformly distributed over n-dimensional $0/1$ vectors having exactly $d$ ones. Then, for every $\alpha>0$ and $\varepsilon\ge0$, we have
\[
\mathcal{L}\left(\sum_{i=1}^nv_i\eta_i,\varepsilon\sqrt{p}\right)\le C_{\ref{anticoncentration estimate}}\varepsilon +\frac{C_{\ref{anticoncentration estimate}}}{\operatorname{CLCD}_{\alpha,\gamma}(v)}+C_{\ref{anticoncentration estimate}}e^{-4\alpha^2d/{n}^2}.
\]
\end{lemma}

\subsection{Operator norm}
For $j\in [n]$, let $Q_n^{(j)}$ denote the $(n-1)\times n$ matrix obtained from $Q_n$ by deleting its $j$-th row. The following estimate is an immediate consequence of \cite[Lemma 2.5]{lianupper}, applied with $m=n-1$.

\begin{lemma}\label{operator norm estimate}
Let $p\in(0,1/2]$ be fixed, and let $d, n$ be positive integers such that $d=pn$. For $j\in[n]$, let $Q_n^{(j)}$ denote the $(n-1)\times n$ matrix obtained from $Q_n$ by deleting its $j$-th row. Then there exist constants $C_{\ref{operator norm estimate}},c_{\ref{operator norm estimate}}>0$ depending only on p such that for all $t\ge C_{\ref{operator norm estimate}}$,
\[
\mathbb{P}\left(\|Q_n^{(j)}-\mathbb{E}Q_n^{(j)}\|\ge t\sqrt{pn}\right)\le 2e^{-c_{\ref{operator norm estimate}}t^2n}.
\]
\end{lemma}
\begin{remark}
Since $\mathbb{E}Q_n^{(j)}=p\mathbf{1}_{n-1}\mathbf{1}_n^T$ and $Q_n^{(j)}\mathbf{1}_n=pn\mathbf{1}_{n-1}$, we have $Q_n^{(j)}-\mathbb{E}Q_n^{(j)}=Q_n^{(j)}P_{\mathbf {1}_n^\perp}$. Consequently,
    \[
   \|Q_n^{(j)}-\mathbb{E}Q_n^{(j)}\|=\|Q_n^{(j)}P_{\mathbf {1}_n^\perp}\|=\|Q_n^{(j)}|_{\mathbf{1}_n^\perp}\|.
    \]
\end{remark}
\section{Invertibility via distance and the lower bound}
Let $X_i:=Q_n^Te_i, i\in[n]$. Then, $X_1,X_2,\dots,X_n$ are independent vectors. For $j\in[n]$, let
$H_j:=\operatorname{span}\{X_i,i \in [n], i\neq j\}$ and $\mathcal{F}_j:=\sigma(X_i,i\in[n],i\neq j)$. Denote by $P_j:\mathbb{R}^n\to H_j^\perp$ the orthogonal projection onto $H_j^\perp$, and $Z_j\in H_j^\perp$ is any $\mathcal{F}_j$-measurable unit vector. On the event $\mathcal{E}_j:=\{\dim(H_j^\perp)=1\}$, $Z_j$ is unique up to sign. By \cite[Remark (2) and (3)]{tran}, $\mathbb{P}(\mathcal{E}_j^c)\le\mathbb{P}(Q_n\ \textit{is singular}) \le 2e^{-c_0n}$, where $c_0$ is a constant that depends only on $p$.

Since $Z_j\in H_j^\perp$ is a unit vector, we have
\[
\operatorname{dist}(M^TX_j,M^TH_j)
\ge
\frac{|\langle Z_j,X_j\rangle|}{\|M^{-1}Z_j\|_2}.
\]
Thus, the small-distance estimate reduces to a small-ball probability estimate for $\langle Z_j,X_j\rangle$ with a random radius proportional to $\|M^{-1}Z_j\|_2$. 
Motivated by \cite[Theorem 4.4]{letwin2026smallest}, we seek a similar upper bound for $\mathbb{E}(\mathbf{1}_{\mathcal{A}_j}\|M^{-1}Z_j\|_2)$ for any $\mathcal{F}_j$-measurable event $\mathcal{A}_j$. However, due to the combinatorial dependence among the entries within each row of $Q_n$, the argument used there does not apply directly. More precisely, we introduce a permutation $\pi \in S_n$ (non random) in the proof. We establish the following average projection estimate.

\subsection{Average projection estimate}

\begin{theorem}\label{average projection estimate}
There exist $C_{\ref{average projection estimate}}\ge1$, depending only on $p$, such that for every $j\in[n]$,
\[
\mathbb{E}(P_j)\preceq \frac{C_{\ref{average projection estimate}}}{n}\mathbf{I}_n.
\]
Consequently, for every deterministic invertible $n\times n$ matrix $M$ and every $\mathcal{F}_j$-measurable event $\mathcal{A}_j$,
\[
\mathbb{E}(\mathbf{1}_{\mathcal{A}_j}\|M^{-1}Z_j\|_2)\le\frac{C_{\ref{average projection estimate}}\|M^{-1}\|_{\mathrm{HS}}}{\sqrt{n}}.
\]
\end{theorem}
\begin{proof}
Let $\pi\in S_n$ be any fixed permutation, and let $\Pi$ denote the corresponding permutation matrix. By the permutation invariance of the distribution $X_i$,
\[
\Pi X_i\stackrel{d}{=} X_i,\quad \forall i\in[n].
\]
Since $X_1,\dots,X_n$ are independent, we have
\[
(\Pi X_i)_{i\in [n]\setminus{\{j\}}}\stackrel{d}{=}( X_i)_{i\in [n]\setminus{\{j\}}}.
\]
Therefore, the orthogonal projection onto $\operatorname{span}\{\Pi X_i,i \in [n], i\neq j\}^{\perp}$ has the same distribution as $P_j$. Since $\Pi$ is orthogonal, this projection is given by $\Pi P_j \Pi^T$. Hence,
\[
P_j\stackrel{d}{=}\Pi P_j\Pi^T.
\]
Taking the expectation and using the fact that $\Pi$ is deterministic, we obtain
\begin{equation}\label{eq:EPj=PiEPjPIT}
   \mathbb{E}(P_j)=\Pi\mathbb{E}(P_j)\Pi^T. 
\end{equation}

Since $\Pi$ is arbitrary, (\ref{eq:EPj=PiEPjPIT}) implies that all diagonal entries of $\mathbb{E}(P_j)$ are equal and all off-diagonal entries are equal. Hence, there exist $\alpha, \beta\in \mathbb{R}$ such that 
\begin{equation}\label{eq:EPj=aP+bP}
\mathbb{E}(P_j)=(\alpha-\beta)\mathbf{I}_n+\beta \mathbf{1}_n\mathbf{1}_n^T=a P_{\mathbf{1}_n^{\perp}}+bP_{\mathbf{1}_n},
\end{equation}

where $a=\alpha-\beta,b=(n-1)\beta+\alpha$. Since $P_j\succeq 0$, we have $a,b\geq 0$.

Next, we estimate upper bounds for $a,b$.

Let $Q_n^{(j)}$ denote the $(n-1)\times n$ matrix obtained from $Q_n$ by deleting its $j$-th row. Then $H_j^\perp=\operatorname{ker}(Q_n^{(j)})$.

Let $e:=n^{-1/2}\mathbf{1}_n$. 
For any unit vector $z\in H_j^{\perp}$, write $z=\gamma e+h$. We have 
\[\gamma=\langle z,e\rangle,\quad h\perp e,\quad\|h\|_2\le1.\]
Since $Q_n^{(j)}z=0$ and $Q_n^{(j)}e=p\sqrt n\,\mathbf1_{n-1}$, we have

\[
|\gamma|p\sqrt{n(n-1)}=\|Q_n^{(j)}h\|_2\leq\|Q_n^{(j)}|_{\mathbf{1}_n^\perp}\|.
\]

Let
\[
\mathcal{B}_j:=\left\{\|Q_n^{(j)}|_{\mathbf{1}_n^\perp}\|\le C_{\ref{operator norm estimate}}\sqrt{pn}\right\}.
\]
By Lemma \ref{operator norm estimate}, $\mathbb{P}(\mathcal{B}_j^c)\le2e^{-c_{\ref{operator norm estimate}}C_{\ref{operator norm estimate}}^2n}$. On $\mathcal B_j$, the preceding bound holds uniformly for all unit vectors $z\in H_j^\perp$. Hence
\begin{equation}\label{eq:gammjupperbound}
\|P_je\|_2=\sup_{\substack{z\in H_j^\perp\\\|z\|_2=1}}|\langle z,e\rangle|
\le\frac{C_{\ref{operator norm estimate}}}{\sqrt{p(n-1)}}.
\end{equation}
By (\ref{eq:EPj=aP+bP}), we have
\[
\mathbb{E}\left(\left\|P_je\right\|^2_2\right)=\left\langle\mathbb{E}(P_j)e,e\right\rangle=b.
\]
Using (\ref{eq:gammjupperbound}) and $\|P_je\|_2^2\le1$, we obtain
\begin{align}\label{bleC1n}
b=\mathbb{E}\left(\left\|P_je\right\|^2_2\right)\le \frac{C_{\ref{operator norm estimate}}^2}{{p(n-1)}} +\mathbb{P}(\mathcal{B}_j^c)\le\frac{C_1}{n},
\end{align}
where $C_1$ depends only on $p$.

Since $\operatorname{tr}(P_j)=\dim (H_j^{\perp})$, we have
\begin{align*}
    \operatorname{tr}(\mathbb{E}(P_j))=\mathbb{E}(\operatorname{tr}(P_j))=\mathbb{E}\left(\operatorname{tr}(P_j) \mathbf{1}_{\{\mathcal{E}_j\}}\right)+\mathbb{E}\left(\operatorname{tr}(P_j) \mathbf{1}_{\{\mathcal{E}_j^c\}}\right)&\le\mathbb{P}(\mathcal{E}_j)+n\mathbb{P}(\mathcal{E}_j^c)\\&\le 1+2ne^{-c_0n}.
\end{align*}

Thus,
\[
a\operatorname{tr}(P_{\mathbf{1}_n^{\perp}})+b\operatorname{tr}(P_{\mathbf{1}_n})=\operatorname{tr}(\mathbb{E}(P_j))\le1+2ne^{-c_0n}.
\]
Since $a,b\ge0$ and $\operatorname{tr}(P_{\mathbf{1}_n^{\perp}})=n-1$, we obtain
\begin{equation}
    a\le\frac{1+2ne^{-c_0n}}{n-1}\le \frac{C_2}{n}.
\end{equation}
Then, we have
\begin{equation}\label{eq:average projection}
\mathbb{E}(P_j)\preceq\frac{C_2}{n} P_{\mathbf{1}_n^{\perp}}+\frac{C_1}{n}P_{\mathbf{1}_n}\preceq\frac{C_{\ref{average projection estimate}}}{n}\mathbf{I}_n,
\end{equation}
where $C_{\ref{average projection estimate}}=\max\{C_1,C_2,1\}$.

The second assertion follows from (\ref{eq:average projection}) by an argument similar to that in the proof of \cite[Theorem 4.4]{letwin2026smallest}. For completeness, we provide the details below.

Since $Z_j$ is a unit vector in $H_j^\perp$ and $P_j$ is the orthogonal projection onto $H_j^\perp$, we have
\[
\mathbf{1}_{\mathcal{A}_j}Z_jZ_j^{T}\preceq P_j.
\]
Thus, we 
\begin{align*}
    \mathbb{E}(\mathbf{1}_{\mathcal{A}_j}\|M^{-1}Z_j\|_2^2)=\operatorname{tr}\left((M^{-1})^TM^{-1}\mathbb{E}(\mathbf{1}_{\mathcal{A}_j}Z_jZ_j^T)\right)&\le\frac{C_{\ref{average projection estimate}}}{n}\operatorname{tr}\left((M^{-1})^TM^{-1}\right)\\&=\frac{C_{\ref{average projection estimate}}\|M^{-1}\|_{\mathrm{HS}}^2}{n}.
\end{align*}
By the Cauchy-Schwarz inequality, we have
\[
\mathbb{E}(\mathbf{1}_{\mathcal{A}_j}\|M^{-1}Z_j\|_2)\le\left(\mathbb{E}(\mathbf{1}_{\mathcal{A}_j}\|M^{-1}Z_j\|_2^2)\right)^{1/2}\le\frac{C_{\ref{average projection estimate}}\|M^{-1}\|_{\mathrm{HS}}}{\sqrt{n}},
\]
where the last follows from $C_{\ref{average projection estimate}}>1$ by enlarging it if necessary. This completes the proof.
\end{proof}
\subsection{Distance estimate}
The key ingredient for controlling the invertibility over incompressible vectors is the following small-distance estimate.
\begin{lemma}\label{distanceupperbound}
There exist constants $C_{\ref{distanceupperbound}},c_{\ref{distanceupperbound}}>0$, depending only on $p$, such that for every $j\in [n]$ and every $\varepsilon>0$,
\[
\mathbb{P}\left(\operatorname{dist} \left(M^T X_j,M^T H_j\right)\le \frac{\varepsilon\sqrt{n}}{\|M^{-1}\|_{\mathrm{HS}}}\right)\le C_{\ref{distanceupperbound}}\varepsilon+2e^{-c_{\ref{distanceupperbound}}n}.
\]
\end{lemma}
\begin{proof}
For $j\in [n]$, let $Q_n^{(j)}$ be obtained from $Q_n$ by deleting its $j$-th row. For all sufficiently large $n$, Proposition \ref{constant estimate} yields 
\begin{equation}\label{eq:constant estimate}
    \mathbb P\left(\inf_{v\in \operatorname{Cons}(\delta_{\ref{constant estimate}},\rho_{\ref{constant estimate}})}\|{Q_n^{(j)} v}\|_2\le C_{\ref{constant estimate}}\sqrt{pn}\right)\le e^{-c_{\ref{constant estimate}}n}.
\end{equation}
By decreasing $\delta_{\ref{constant estimate}},\rho_{\ref{constant estimate}}$, if necessary, we may assume that they are sufficiently small, and keep them fixed throughout the remainder of the proof. Adapting \cite[Proposition 2.13]{tran} to the case $d=pn$, we obtain constants $\mu,\gamma,c,c'\in(0,1)$, depending only on $p$, such that for all sufficiently large $n$,
\begin{equation}\label{eq:random normal}
\mathbb{P}\left(\exists v\in \mathbb{S}^{n-1}\setminus\operatorname{Cons}(\delta_{\ref{constant estimate}},\rho_{\ref{constant estimate}}): Q_n^{(j)}v=0\ \text{and}\ \operatorname{CLCD}_{\mu n,\gamma}(v)\le e^{cn}\right)\le e^{-c'n}.
\end{equation}
Define the $\mathcal{F}_j$-measurable event
\[
\mathcal{A}_j:=\left\{Z_j\in\mathbb{S}^{n-1}\setminus\operatorname{Cons}(\delta_{\ref{constant estimate}},\rho_{\ref{constant estimate}})\ \text{and}\ \operatorname{CLCD}_{\mu n,\gamma}(Z_j)\ge e^{cn}\right\}.
\]
Since $Q_n^{(j)}Z_j=0$, it follows from (\ref{eq:constant estimate}) and (\ref{eq:random normal}) that,
\begin{equation}\label{eq:Ajcbound}
    \mathbb{P}(\mathcal{A}_j^c)\le e^{-c'n}+e^{-c_{\ref{constant estimate}}n}\le e^{-c''n},
\end{equation}
for all sufficiently large $n$, where $c''>0$ depends only on $p$.

By Lemma \ref{nonconstanseparated}, for $v\in \mathbb{S}^{n-1}\setminus\operatorname{Cons}(\delta_{\ref{constant estimate}},\rho_{\ref{constant estimate}})$, we have
\[
\|\mathrm{D}(v)\|_2\ge\frac{\delta_{\ref{constant estimate}}\rho_{\ref{constant estimate}}\sqrt{n}}{8\sqrt{2}}.
\]
Take $b=\delta_{\ref{constant estimate}}\rho_{\ref{constant estimate}}/8\sqrt{2}$ and $ \alpha=\mu n$. Since $X_j$ is independent of $\mathcal F_j$ and $Z_j$ is $\mathcal F_j$-measurable, conditioning on the event $\mathcal{A}_j$, we apply Lemma \ref{anticoncentration estimate} to see that, for every $\varepsilon\ge0$,
\begin{align}\label{eq:ZjXjupperbound1}
\mathbb{P}\left(|\langle Z_j,X_j\rangle |\le\varepsilon\,\middle|\,\mathcal{F}_j\right)\le C_1\left(\varepsilon+e^{-c_1 n}\right),
\end{align}
where $C_1,c_1>0$ depend only on $p$.
By using (\ref{eq:ZjXjupperbound1}) and Theorem \ref{average projection estimate}, we obtain
\begin{align}\label{eq:ZjXjupperbound2}
\mathbb{P}&\left(\left\{|\langle Z_j,X_j\rangle|\le\frac{\varepsilon\sqrt{n}\|M^{-1}Z_j\|_2}{\|M^{-1}\|_{\mathrm{HS}}}\right\}\cap\mathcal{A}_j\right)\notag\\&\qquad\qquad=\mathbb{E}\left( \mathbf{1}_{\{\mathcal{A}_j\}}\mathbb{P}\left(|\langle Z_j,X_j\rangle|\le\frac{\varepsilon\sqrt{n}\|M^{-1}Z_j\|_2}{\|M^{-1}\|_{\mathrm{HS}}} \,\middle|\,\mathcal{F}_j\right)\right)\notag\\
&\qquad\qquad\le C_1\frac{\varepsilon\sqrt{n}}{\|M^{-1}\|_{\mathrm{HS}}}\mathbb{E}\left(\mathbf{1}_{\{\mathcal{A}_j\}}\|M^{-1}Z_j\|_2\right)+C_1e^{-c_1n}\notag\\
&\qquad\qquad\le C_{\ref{average projection estimate}}C_1\varepsilon+C_1e^{-c_1n}.
\end{align}
Since
\[
\operatorname{dist}(M^TX_j,M^TH_j)
\ge
\frac{|\langle Z_j,X_j\rangle|}{\|M^{-1}Z_j\|_2},
\]
using (\ref{eq:Ajcbound}) and (\ref{eq:ZjXjupperbound2}) gives
\begin{align}
\mathbb{P}\left(\operatorname{dist} \left(M^T X_j,M^T H_j\right)\le \frac{\varepsilon\sqrt{n}}{\|M^{-1}\|_{\mathrm{HS}}}\right)&\le\mathbb{P}\left(\frac{|\langle Z_j,X_j\rangle|}{\|M^{-1}Z_j\|_2}\le \frac{\varepsilon\sqrt{n}}{\|M^{-1}\|_{\mathrm{HS}}}\right)\notag\\
&\le C_{\ref{average projection estimate}}C_1\varepsilon+C_1e^{-c_1n}+e^{-c''n}\notag\\
&\le C_{\ref{distanceupperbound}}\varepsilon+2e^{-c_{\ref{distanceupperbound}}n}.
\end{align}
The last inequality holds for all sufficiently large $n$. By decreasing $c_{\ref{distanceupperbound}}>0$, if necessary, the remaining finitely many values of $n$ can be absorbed into the exponential term.
\end{proof}
Combining the above distance estimate with the lemma of Rudelson and Vershynin \cite[Lemma 3.5]{thelittle}, we obtain the desired probability bound for incompressible vectors.
\begin{align}\label{incompressible estimate}
\mathbb{P}&\left(\inf_{x\in\operatorname{Incomp(\delta_{\ref{compressible estimate}},\rho_{\ref{compressible estimate}})}}\|M^TQ_n^Tx\|_2\le\frac{\varepsilon \rho_{\ref{compressible estimate}}}{2\|M^{-1}\|_{\mathrm{HS}}}\right)\notag\\&\le\frac{1}{\delta_{\ref{compressible estimate}}n}\sum_{j=1}^n\mathbb{P}\left(\operatorname{dist} \left(M^T X_j,M^T H_j\right)\le \frac{\varepsilon\sqrt{n}}{\|M^{-1}\|_{\mathrm{HS}}}\right)\notag\\
    &\le\frac{1}{\delta_{\ref{compressible estimate}}}\left(C_{\ref{distanceupperbound}}\varepsilon+2e^{-c_{\ref{distanceupperbound}}n}\right).
\end{align}
\begin{remark}
Here the factor $1/2$ arises from the difference between our definition of sparse vectors and that in \cite[Definition 3.2]{thelittle}. The former is restricted to $\mathbb{S}^{n-1}$, while the latter is defined in the whole $\mathbb{R}^n$. The corresponding notions of incompressibility are equivalent up to an absolute constant, and in particular,
\[
\operatorname{Incomp}(\delta,\rho)\subseteq \operatorname{Incomp}_{\mathrm{RV}}(\delta,\frac{\rho}{2}).
\]
Here, $\operatorname{Incomp}_{\mathrm{RV}}$ denotes the incompressible set defined by Rudelson and Vershynin \cite{thelittle}. Hence \cite[Lemma 3.5]{thelittle} can be applied with the parameter $\rho/2$.
\end{remark}
\subsection{Proof of Theorem 1.1}
\begin{proof}
It suffices to prove the theorem for all sufficiently large $n\geq n_0$ and sufficiently small $0\le\varepsilon\le \varepsilon_0$, where $n_0,\varepsilon_0$ are constants depending only on $p$. The remaining cases can be absorbed by adjusting the constants in the statement. 

By using Proposition \ref{compressible estimate} and (\ref{incompressible estimate})
\begin{align}
    \mathbb{P}\left( s_n(Q_nM) \le \frac{\varepsilon}{\|M^{-1}\|_{\mathrm{HS}}} \right)&=\mathbb{P}\left( s_n(M^TQ_n^T) \le \frac{\varepsilon}{\|M^{-1}\|_{\mathrm{HS}}} \right)\notag\\
    &\le \mathbb{P}\left(\inf_{x\in\operatorname{Comp}(\delta_{\ref{compressible estimate}},\rho_{\ref{compressible estimate}})}\|M^TQ_n^Tx\|_2 \le \frac{C_{\ref{compressible estimate}} \sqrt{n}}{\|M^{-1}\|_{\mathrm{HS}}} \right)\notag\\&+\mathbb{P}\left(\inf_{x\in\operatorname{Incomp(\delta_{\ref{compressible estimate}},\rho_{\ref{compressible estimate}})}}\|M^TQ_n^Tx\|_2\le\frac{\varepsilon}{\|M^{-1}\|_{\mathrm{HS}}}\right)\notag\\
    &\le 2e^{-c_{\ref{compressible estimate}}n}+\frac{1}{\delta_{\ref{compressible estimate}}}\left(C_{\ref{distanceupperbound}}\frac{2\varepsilon}{\rho_{\ref{compressible estimate}}}+2e^{-c_{\ref{distanceupperbound}}n}\right)\notag\\
    &\le C_{\ref{theorem 1.1}}\varepsilon +2e^{-c_{\ref{theorem 1.1}}n}.\notag
\end{align}
 This completes the proof of Theorem \ref{theorem 1.1}.
\end{proof}
\section{Random permutations and the upper bound}
Let $\pi\sim \operatorname{Unif}(S_n)$ be a uniformly random permutation in $S_n$, independent of $Q_n$, and let $\Pi$
denote the corresponding random permutation matrix. 
By the permutation invariance of $Q_n$,  
\[
Q_n\Pi^T\stackrel{d}{=}Q_n,
\]
and hence
\[
s_n(Q_n\Pi^TM)\stackrel{d}{=}s_n(Q_nM).
\]
Let $u=u(Q_n)\in \mathbb{S}^{n-1}$ be a right singular vector of $Q_n$ corresponding to its smallest singular value $s_n(Q_n)$. We have 
\[
s_n(Q_n\Pi^TM)\le\frac{\|Q_n\Pi^TMM^{-1}\Pi u\|_2}{\|M^{-1}\Pi u\|_2}=\frac{s_n(Q_n)}{\|M^{-1}\Pi u\|_2}.
\]
Consequently, to estimate the upper tail of $s_n(Q_nM)$, it suffices to control the upper tail of $s_n(Q_n)$ at the scale $1/\sqrt{n}$ and the lower tail of $\|M^{-1}\Pi u\|_2$ at the scale $\|M^{-1}\|_{\mathrm{HS}}/{\sqrt{n}}$. The former estimate has already been established in \cite[Theorem 1.1]{lianupper}. Therefore, we focus on the latter.
\subsection{Matching and random permutations via Rademacher swaps}
Unlike in the Gaussian case \cite[Theorem 5.1]{letwin2026smallest}, a direct Hanson-Wright argument is not available. Instead, we introduce independent Rademacher swaps into a random permutation to recover sufficient independent randomness for its application. Before presenting the construction, we recall the notion of a matching. A \emph{matching} on $[n]$ is a collection $\mathcal{M}\subseteq \binom{[n]}{2}$ of pairwise disjoint two-element subsets. Equivalently, it is a collection of edges, no two of which share a common vertex. To construct the required random permutation, we need a matching that captures the projected Hilbert-Schmidt norm. The following lemma provides such a matching.
\begin{lemma}\label{matchingwithlargeweight}
Let $A$ be any $n\times n$ deterministic matrix, set $w_{ij}:=\|A(e_i-e_j)\|_2^2,1\le i<j\le n$. Then, there exists a matching $\mathcal{M}\subseteq\binom{[n]}{2}$ with $|\mathcal{M}|=\lfloor n/2\rfloor$ such that
\[
\sum_{\{i,j\}\in \mathcal{M}}w_{ij}\ge\|AP_{\mathbf 1_n^\perp}\|_{\mathrm{HS}}^2
\] 
\end{lemma}
\begin{proof}
Note that 
\begin{align*}
\sum_{1\le i<j\le n}w_{ij}&=\sum_{1\le i<j\le n}\operatorname{tr}(A(e_i-e_j)(e_i-e_j)^TA^T)\\
&=\operatorname{tr}\left(A\sum_{1\le i<j\le n}{(e_i-e_j)(e_i-e_j)^T}A^T\right)\\
&=n\operatorname{tr}(A P_{\mathbf{1}_{n}^{\perp}}A^T)=n\|A P_{\mathbf{1}_{n}^{\perp}}\|_{\mathrm{HS}}^2.
\end{align*}
Let $\mathcal{M'}$ be uniform over all matchings on $[n]$ of cardinality $\lfloor n/2\rfloor$. By symmetry, we have
\[
\left\lfloor\frac{n}{2}\right\rfloor=\mathbb{E}(|\mathcal{M'}|)=\binom{n}{2}\mathbb{P}(\{i,j\}\in\mathcal{M'}).
\]
Thus, we have
\[
\mathbb{P}(\{i,j\}\in\mathcal{M'})=
\begin{cases}
    \dfrac{1}{n-1},& n \text{ is even},\\
    \quad\dfrac{1}{n},& n \text{ is odd}.
\end{cases}
\]
Then we observe
\[
\mathbb{E}\left(\sum_{\{i,j\}\in \mathcal{M'}}w_{ij}\right)=
\begin{cases}
    \dfrac{n}{n-1}\|A P_{\mathbf{1}_{n}^{\perp}}\|_{\mathrm{HS}}^2,& n \text{ is even},\\
    \quad\|A P_{\mathbf{1}_{n}^{\perp}}\|_{\mathrm{HS}}^2,& n \text{ is odd}.
\end{cases}
\]
Then, there exists a realization of $\mathcal{M}'$ that satisfies $\sum_{\{i,j\}\in \mathcal{M}}w_{ij}\ge\|AP_{\mathbf 1_n^\perp}\|_{\mathrm{HS}}^2$.
\end{proof}
The previous lemma gives a matching with large total weight. We next show that, after a uniform random permutation, a positive fraction of this weight crosses between two large disjoint coordinate sets, except with an exponentially small probability.
\begin{lemma}\label{Weightcaptured}
Fix $\beta\in(0,1/2]$. Let $I,J\subset[n]$ be disjoint sets with $|I|,|J|\ge\beta n$, and let 
\[
\mathcal{M}=\{m_i=\{r_i,s_i\}:1\le i\le k\}
\]
 be a matching on $[n]$ with $1\le k\le\beta n/8$. Assign deterministic weights $w_i\in[0,w_{\max}]$, where $w_{\max}>0$, and let $\pi\sim \operatorname{Unif}(S_n)$ be a uniformly random permutation. Define
\[
X_{i}:=\mathbf{1}_{\{|\{\pi(r_i),\pi(s_i)\}\cap I|=|\{\pi(r_i),\pi(s_i)\}\cap J|=1\}}.
\]
Then there exists a constant $c_{\ref{Weightcaptured}}>0$, depending only on $\beta$, such that 
\[
\mathbb{P}\left(\sum_{i=1}^kw_iX_i\le c_{\ref{Weightcaptured}}\sum_{i=1}^k w_i\right)\le\exp\left(-c_{\ref{Weightcaptured}}\frac{\sum_{i=1}^{k}w_i}{w_{\max}}\right).
\]
\end{lemma}
\begin{proof}
For $1 \le i\le k$, let $B_i:=\{\pi(r_i),\pi(s_i)\}$, $\mathcal{G}_{i-1}:=\sigma(B_1,\dots,B_{i-1})$, and $\mathcal{G}_0:=\{\Omega,\varnothing\}$, and set  
\[
U_i:=[n]\setminus\bigcup_{j<i}B_j,\quad 1<i\le k,\quad U_1=[n].
\]

Conditional on $\mathcal{G}_{i-1}$, the pair $B_i$ is uniformly distributed over all two-element subsets of $U_i$. Since $|U_i|=n-2(i-1)$, for every $b\in\binom{[n]}{2}$,
\[
\mathbb{P}\left(B_i=b \,\middle|\,\mathcal{G}_{i-1}\right)=\frac{\mathbf{1}_{\{b\subset U_i\}}}{\binom{|U_i|}{2}}.
\]
Let $I_i:=I\cap U_i$ and $J_i=J\cap U_i$, we obtain
\begin{equation}\label{probability bound}
    \mathbb{P}\left(X_i=1\,\middle|\,\mathcal{G}_{i-1}\right)=\frac{2|I_i||J_i|}{|U_i|(|U_i|-1)}\ge\frac{2(\beta n/2)^2}{n^2}=\frac{\beta^2}{2},
\end{equation}
where the inequality follows from  $|U_i|\le n$ and $|I_i|,|J_i|\ge \beta n -\beta n/4\ge \beta n/2$.

Following (\ref{probability bound}), we have 
\begin{align}\label{conditionexpectionbound}
 \mathbb{E}\left(e^{-w_iX_i/w_{\max}}\,\middle|\,\mathcal{G}_{i-1}\right)\le 1-\frac{\beta^2}{2}\left(1-e^{-w_i/w_{\max}}\right)&\le 1-\frac{(1-e^{-1})\beta^2w_i}{2w_{\max}}\notag\\&\le\exp\left(-\frac{(1-e^{-1})\beta^2w_i}{2w_{\max}}\right),
\end{align}
where we used  $1-e^{-x}\ge (1-e^{-1})x$ for $x\in[0,1]$ and $1-x\le e^{-x}$.

By (\ref{conditionexpectionbound}),
\begin{align}
\mathbb{E}\left(e^{-\sum_{i=1}^kw_iX_i/w_{\max}}\right)
&=\mathbb{E}\left[
\exp\left(
-\frac{1}{w_{\max}}\sum_{i=1}^{k-1}w_iX_i
\right)
\mathbb{E}\left(
e^{-w_kX_k/w_{\max}}
\,\middle|\,
\mathcal{G}_{k-1}
\right)
\right]\notag\\
&\le
\exp\left(
-\frac{(1-e^{-1})\beta^2w_k}{2w_{\max}}
\right)
\mathbb{E}\left[
\exp\left(
-\frac{1}{w_{\max}}\sum_{i=1}^{k-1}w_iX_i
\right)
\right]\notag\\
&\le
\exp\left(
-\frac{(1-e^{-1})\beta^2}{2w_{\max}}
\sum_{i=1}^k w_i
\right).\notag
\end{align}
Thus
\begin{align}
\mathbb{P}\left(\sum_{i=1}^kw_iX_i\le c_{\ref{Weightcaptured}}\sum_{i=1}^k w_i\right)&\le\mathbb{E}\left(e^{-\sum_{i=1}^kw_iX_i/w_{\max}}\right)\exp\left(c_{\ref{Weightcaptured}}\frac{\sum_{i=1}^k w_i}{w_{\max}}\right)\notag\\&\le\exp\left(-c_{\ref{Weightcaptured}}\frac{\sum_{i=1}^{k}w_i}{w_{\max}}\right),\notag
\end{align}
where $c_{\ref{Weightcaptured}}=(1-e^{-1})\beta^2/4$. This completes the proof.
\end{proof}
We next recall the Rademacher swaps coupling. Independent swaps within disjoint pairs preserve the uniform permutation law and leave the unordered image pairs unchanged. This construction is stated explicitly in \cite[Fact 3.6 in version 4]{tran}, a related coupling for the Boolean slice is given in \cite[Fact 2.7]{kwan2019anticoncentration}. We record the formulation needed here for completeness.
\begin{lemma}\label{independentswapsondisjointpairs}
    Let $\mathcal{M}=\{m_i=\{r_i,s_i\}:1\le i\le N\}$ be a deterministic matching on $[n]$, where $N\le\lfloor n/2\rfloor$. Let $\pi\sim \operatorname{Unif}(S_n)$, and let  $\varepsilon_1,\dots,\varepsilon_N$ be independent Rademacher random variables, independent of $\pi$. For $\varepsilon=(\varepsilon_1,\dots,\varepsilon_N)$, define $\tau_\varepsilon\in S_n$ by
    \[\tau_{\varepsilon}(r_i)=\begin{cases}
    r_i,&\varepsilon_i=1,\\
    s_i,& \varepsilon_i=-1,
\end{cases}
\tau_{\varepsilon}(s_i)=\begin{cases}
    s_i,&\varepsilon_i=1,\\
    r_i,& \varepsilon_i=-1,
\end{cases}
\ 1\le i\le N, \text{and}\ \tau_\varepsilon(k)=k,\ k\notin \bigcup_{i=1}^N\{r_i,s_i\}. 
    \]
Set $\Tilde{\pi}:=\pi\circ \tau_{\varepsilon}$. Then we have the following results:
\begin{itemize}
    \item $\Tilde{\pi}\sim\operatorname{Unif}(S_n)$;
    \item For every $1\le i\le N$, we have $\{\Tilde{\pi}(r_i),\Tilde{\pi}(s_i)\}=\{\pi(r_i),\pi(s_i)\}$;
    \item For any vector $v\in \mathbb{R}^n$, we have 
    \[
    v_{\Tilde{\pi}(r_i)}e_{r_i}+v_{\Tilde{\pi}(s_i)}e_{s_i}=\frac{v_{\pi(r_i)}+v_{\pi(s_i)}}{2}(e_{r_i}+e_{s_i})+\varepsilon_i\frac{v_{\pi(r_i)}-v_{\pi(s_i)}}{2}(e_{r_i}-e_{s_i}).
    \]
\end{itemize}
\end{lemma}
\begin{proof}
For each fixed $\varepsilon$, the map $\tau_\varepsilon$ is a permutation of $[n]$. Consequently, the map $\pi\mapsto\pi\circ\tau_\varepsilon$ is a bijection from $S_n$ to itself. Therefore, for every $\sigma \in S_n$,
\[
\mathbb{P}(\Tilde{\pi}=\sigma\,|\,\varepsilon)=\mathbb{P}\left(\pi=\sigma\circ\tau_{\varepsilon}^{-1}\,\middle|\,\varepsilon\right)=\frac{1}{n!},
\]
since $\varepsilon$ is independent of $\pi$ and $\pi \sim \operatorname{Unif}(S_n)$. Hence,
\[
\Tilde{\pi}\sim\operatorname{Unif}(S_n).
\]
For $1\le i\le N$, the permutation $\tau_\varepsilon$ either fixes $r_i,s_i$ or exchanges them. Hence, the unordered pair is preserved
\[
\{\Tilde{\pi}(r_i),\Tilde{\pi}(s_i)\}=\{\pi(r_i),\pi(s_i)\}.
\]
Moreover, the sign $\varepsilon_i$ records whether the two entries are kept or exchanged. Hence
\[
v_{\Tilde{\pi}(r_i)}e_{r_i}+v_{\Tilde{\pi}(s_i)}e_{s_i}=\frac{v_{\pi(r_i)}+v_{\pi(s_i)}}{2}(e_{r_i}+e_{s_i})+\varepsilon_i\frac{v_{\pi(r_i)}-v_{\pi(s_i)}}{2}(e_{r_i}-e_{s_i}).
\]
This completes the proof.
\end{proof}
\subsection{Permuted norm estimate}
To estimate the lower tail of $\|M^{-1}\Pi u\|_2$, we focus on non-almost-constant vectors. Combining Proposition \ref{constant estimate} with the upper bound for $s_n(Q_n)$, controls the probability that the least right singular vector $u$ is almost constant. It therefore remains to consider $u\in\mathbb{S}^{n-1}\setminus \operatorname{Cons}(\delta,\rho)$, for which we establish the following estimate.
\begin{lemma}\label{permutednormestimate}
Let $\delta,\rho\in(0,1)$. There exist constants
$C_{\ref{permutednormestimate}},c_{\ref{permutednormestimate}}>0$
and $n_{\ref{permutednormestimate}}\in\mathbb{N}$, depending only on
$\delta$ and $\rho$, such that the following holds. Let $n\ge n_{\ref{permutednormestimate}}$, and let $A$ be a deterministic nonzero $n\times n$ matrix. Let $\pi\sim\operatorname{Unif}(S_n)$ with the corresponding permutation
matrix $\Pi$. Then for any fixed $v\in \mathbb{S}^{n-1}\setminus \operatorname{Cons}(\delta,\rho)$,
\[
\mathbb{P}\left(
\|A\Pi v\|_2
\le
c_{\ref{permutednormestimate}}
\frac{\|AP_{\mathbf{1}_n^\perp}\|_{\mathrm{HS}}}{\sqrt n}
\right)
\le
C_{\ref{permutednormestimate}}
\exp\left(
-c_{\ref{permutednormestimate}}
\frac{\|AP_{\mathbf{1}_n^\perp}\|_{\mathrm{HS}}^2}
{\|A\|^2}
\right).
\]
\end{lemma}
\begin{proof}
By Lemma \ref{matchingwithlargeweight}, there exists a deterministic matching
$\mathcal{M}=\{m_i=\{r_i,s_i\}:1\le i\le N\}$, $N=\lfloor n/2\rfloor$, such that 
\[
\sum_{i=1}^Nw_{i}\ge\|AP_{\mathbf 1_n^\perp}\|_{\mathrm{HS}}^2,
\]
where $w_i=\|A(e_{r_i}-e_{s_i})\|_2^2\le2\|A\|^2$. 

We enumerate the pairs in $\mathcal{M}$ so that $w_1\ge w_2\dots\ge w_N$. For definiteness, write $m_i=\{r_i,s_i\}$ with $r_i>s_i$, while still regarding $m_i$ as an unordered pair.

Let
\[
\beta=\frac{\delta}{8},\quad \gamma=\min\left\{\frac{\beta}{8}, \frac{1}{16}\right\},\quad k=\lfloor \gamma n\rfloor.
\]
Enlarging $n_{\ref{permutednormestimate}}$, if necessary, we may assume that $n\ge2/\gamma$. Hence, $\gamma n/2\le k\le N$. Then we obtain
\begin{equation}\label{Wkbound}
    W_k=\sum_{i=1}^k w_i\ge\frac{k}{N}\sum_{i=1}^Nw_i\ge\gamma\|AP_{\mathbf{1}_{n}^\perp}\|_{\mathrm{HS}}^2.
\end{equation}
 Let $v\in\mathbb{S}^{n-1}\setminus\operatorname{Cons(\delta,\rho)}$. By Lemma \ref{nonconstanseparated}, there exist disjoint subsets $I,J\subset[n]$ such that
 \begin{equation}\label{coordinateseparate}
      |I|,|J|\ge \frac{\delta n}{8}=\beta n,
\qquad \text{and} \qquad
\frac{\rho}{\sqrt{2n}}
\le |v_i-v_j|
\le \frac{6}{\sqrt{\delta n}}
\quad
\forall i\in I,\ \forall j\in J.
 \end{equation}
 Let $\pi \sim \operatorname{Unif}(S_n)$, and let $\varepsilon_1,\dots,\varepsilon_N$ be independent Rademacher random variables, independent of $\pi$. Let $\Tilde{\pi}=\pi\circ\tau_\varepsilon$. By Lemma \ref{independentswapsondisjointpairs}, $\Tilde{\pi}\sim \operatorname{Unif}(S_n)$, since the uniform distribution on $S_n$ is invariant under inversion, both $\Tilde{\Pi}$ and $\Tilde{\Pi}^T$ have the same distribution as $\Pi$. 
 
 Let $X_{i}=\mathbf{1}_{\{|\{{\pi}(r_i),{\pi}(s_i)\}\cap I|=|\{{\pi}(r_i),{\pi}(s_i)\}\cap J|=1\}}$,
 \[\mathcal{G}:=\{i\in[k]:X_i=1\},\quad \text{and}\quad \mathcal{E}:=\left\{\sum_{i\in \mathcal{G} }w_i\ge c_{\ref{Weightcaptured}}W_k\right\}.\]
 Applying Lemma \ref{Weightcaptured} with $w_{\max}=2\|A\|^2$, we obtain
 \begin{equation}\label{eventbound}
\mathbb{P}(\mathcal{E}^c)\le\mathbb{P}\left(\sum_{i=1}^kw_iX_i\le c_{\ref{Weightcaptured}}W_k\right)\le\exp \left(-c_{\ref{Weightcaptured}}\frac{W_k}{w_{\max}}\right)\le\exp\left(-\frac{c_{\ref{Weightcaptured}}\gamma\|AP_{\mathbf{1}_{n}^\perp}\|_{\mathrm{HS}}^2}{2\|A\|^2}\right).
\end{equation}
Fix a realization of $\pi$ for which $\mathcal{E}$ holds. Then $\mathcal{G}$ is a deterministic index set. We further fix a realization of $(\varepsilon_i)_{i\in[N]\setminus\mathcal{G}}$. Then the remaining variables $(\varepsilon_i)_{i\in \mathcal{G}}$ are still independent Rademacher random variables. For $i\in\mathcal{G}$, let $\xi_i:=\varepsilon_{i}\operatorname{sgn}(v_{\pi(r_i)}-v_{\pi(s_i)})$, $(\xi_i)_{i\in\mathcal{G}}$ are independent Rademacher random variables, under the above condition. 
Let
\[
d_i:=\frac{|v_{\pi(r_i)}-v_{\pi(s_i)}|}{2}, 
\]
by (\ref{coordinateseparate}), we have
\begin{equation}\label{dibound}
    \frac{\rho}{2\sqrt{2n}}\le d_i \le\frac{3}{\sqrt{\delta n}}.
\end{equation}

Following Lemma \ref{independentswapsondisjointpairs}, for $i\in\mathcal{G}$,
\begin{align*}
v_{\Tilde{\pi}(r_i)}e_{r_i}+v_{\Tilde{\pi}(s_i)}e_{s_i}&=\frac{v_{\pi(r_i)}+v_{\pi(s_i)}}{2}(e_{r_i}+e_{s_i})+\varepsilon_i\frac{v_{\pi(r_i)}-v_{\pi(s_i)}}{2}(e_{r_i}-e_{s_i})\\
&=\frac{v_{\pi(r_i)}+v_{\pi(s_i)}}{2}(e_{r_i}+e_{s_i})+\xi_id_i(e_{r_i}-e_{s_i}).
\end{align*}
Thus,
\[
A\Tilde{\Pi}^Tv=A\sum_{i=1}^nv_{\Tilde{\pi}(i)}e_i=a+\sum_{i\in\mathcal{G}}\xi_id_iA(e_{r_i}-e_{s_i}),
\]
where $a$ is a deterministic vector and the only remaining randomness is in the $(\xi_i)_{i\in\mathcal{G}}$.

Let $P$ be the orthogonal projection matrix onto $\operatorname{span}\{a\}^{\perp}$, then we obtain
\[
\left\|a+\sum_{i\in\mathcal{G}}\xi_id_iA(e_{r_i}-e_{s_i})\right\|_2\ge \left\|\sum_{i\in\mathcal{G}}\xi_id_iPA(e_{r_i}-e_{s_i})\right\|_2.
\]
Let $B$ be the matrix whose columns are $d_iPA(e_{r_i}-e_{s_i})$, indexed by $i\in\mathcal{G}$, and let $\xi:=(\xi_i)_{i\in\mathcal{G}}$ be the corresponding random vector. Note that $\|B\xi\|_2\le\|A\Tilde{\Pi}^Tv\|_2$.

By (\ref{Wkbound}) and (\ref{dibound}),
\begin{align}\label{Atrbound}
\operatorname{tr}\left(A\sum_{i\in\mathcal{G}}d_i^2(e_{r_i}-e_{s_i})(e_{r_i}-e_{s_i})^TA^T\right)&=\sum_{i\in\mathcal{G}}d_i^2\|A(e_{r_i}-e_{s_i})\|_2^2=\sum_{i\in\mathcal{G}}d_i^2w_i\notag\\&\ge\frac{\rho^2c_{\ref{Weightcaptured}}\gamma}{8n}\|AP_{\mathbf{1}_n^\perp}\|_{\mathrm{HS}}^2.
\end{align}
Since the pair $m_i=\{r_i,s_i\}$ are disjoint, for distinct $i,j\in\mathcal{G}$, $\langle e_{r_i}-e_{s_i},e_{r_j}-e_{s_j}\rangle=0$. Thus, by (\ref{dibound}),
\[
\left\|\sum_{i\in\mathcal{G}}d_i^2(e_{r_i}-e_{s_i})(e_{r_i}-e_{s_i})^T\right\|\le \frac{9}{\delta n}\left\|\sum_{i\in\mathcal{G}}(e_{r_i}-e_{s_i})(e_{r_i}-e_{s_i})^T\right\|\le\frac{18}{\delta n},
\]
then we have
\begin{equation}\label{Aopbound}
\|B^{T}B\|=\|BB^{T}\|\le\left\|A\sum_{i\in\mathcal{G}}d_i^2(e_{r_i}-e_{s_i})(e_{r_i}-e_{s_i})^TA^T\right\|\le\frac{18}{\delta n} \|A\|^2.
\end{equation}
Set 
\[
r_0(\delta,\rho)=\max\left\{2,\frac{288}{\delta\rho^2c_{\ref{Weightcaptured}}\gamma}\right\},
\]
we first consider the case
\[
\frac{\|AP_{\mathbf{1}_n^\perp}\|_{\mathrm{HS}}^2}
{\|A\|^2}\ge r_0(\delta,\rho).
\]
The remaining case will be handled by adjusting the constants.  Combining (\ref{Atrbound}) and (\ref{Aopbound}), we obtain
\begin{align}\label{Btrbound}
\mathbb{E}\left(\xi^TB^TB\xi\right)&=\operatorname{tr}\left(B^TB\right)\notag\\&=\operatorname{tr}\left(P A\sum_{i\in\mathcal{G}}d_i^2(e_{r_i}-e_{s_i})(e_{r_i}-e_{s_i})^TA^TP\right)\notag\\
&=\operatorname{tr}\left( A\sum_{i\in\mathcal{G}}d_i^2(e_{r_i}-e_{s_i})(e_{r_i}-e_{s_i})^TA^TP\right)\notag\\
&\ge\operatorname{tr}\left(A\sum_{i\in\mathcal{G}}d_i^2(e_{r_i}-e_{s_i})(e_{r_i}-e_{s_i})^TA^T\right)-\left\|A\sum_{i\in\mathcal{G}}d_i^2(e_{r_i}-e_{s_i})(e_{r_i}-e_{s_i})^TA^T\right\|\notag\\
&\ge \frac{\rho^2c_{\ref{Weightcaptured}}\gamma}{8n}\|AP_{\mathbf{1}_n^\perp}\|_{\mathrm{HS}}^2-\frac{18}{\delta n} \|A\|^2\notag\\
&\ge \frac{c_1}{n}\|AP_{\mathbf{1}_n^\perp}\|_{\mathrm{HS}}^2.
\end{align}
The first inequality follows from the fact that $P=\mathbf{I}_n-aa^T/\|a\|_2^2$ when \(a\neq0\), and $P=\mathbf{I}_n$ otherwise. Here $c_1=\rho^2c_{\ref{Weightcaptured}}\gamma/16$.

Since $B^TB$ is positive semidefinite,
\[
\|B^TB\|_{\mathrm{HS}}^2\le\|B^TB\|\operatorname{tr}(B^TB).
\]
Hence, by using (\ref{Aopbound}),  (\ref{Btrbound}) and the Hanson-Wright inequality \cite[Theorem 6.2.2]{highdimensional}, we obtain
\begin{align}
\mathbb{P}_{\xi}\left(\|B\xi\|_2\le \frac{\sqrt{c_1}\|AP_{\mathbf{1}_n^\perp}\|_{\mathrm{HS}}}{\sqrt{2n}}\right)&\le \mathbb{P}_{\xi}\left(|\xi^TB^TB\xi-\mathbb{E}(\xi^TB^TB\xi)|\ge\frac{\operatorname{tr}(B^{T}B)}{2}\right)\notag\\
&\le 2\exp\left(-c\frac{\operatorname{tr}(B^TB)}{\|B^TB\|}\right)\notag\\
&\le 2\exp\left(-c_2\frac{\|AP_{\mathbf{1}_n^\perp}\|_{\mathrm{HS}}^2}{\|A\|^2}\right),\notag
\end{align}
where the constant $c_2>0$ depends only on $\delta,\rho$. 
Then we obtain
\begin{align}\label{conditionnormbound}
&\mathbb{P}\left(\left\{\|A\Tilde{\Pi}^T v\|_2
\le
\frac{\sqrt{c_1}\|AP_{\mathbf{1}_n^\perp}\|_{\mathrm{HS}}}{\sqrt{2n}}\right\}\cap\mathcal{E}\right)\notag\\&\le\mathbb{E}_{\pi}\left(\mathbf{1}_{\{\mathcal{E}\}}\mathbb{E}_{(\varepsilon_i)_{i\in[N]\setminus\mathcal{G}}}\left(\mathbb{P}_{(\varepsilon_i)_{i\in\mathcal{G}}}\left(\|B\xi\|_2\le \frac{\sqrt{c_1}\|AP_{\mathbf{1}_n^\perp}\|_{\mathrm{HS}}}{\sqrt{2n}}\right) \right)\right)\notag\\
&\le 2\exp\left(-c_2\frac{\|AP_{\mathbf{1}_n^\perp}\|_{\mathrm{HS}}^2}{\|A\|^2}\right).
\end{align}

Combining this conditional argument (\ref{conditionnormbound}) with (\ref{eventbound}), yields
\begin{align}
\mathbb{P}\left(
\|A\Tilde{\Pi}^T v\|_2
\le
\frac{\sqrt{c_1}\|AP_{\mathbf{1}_n^\perp}\|_{\mathrm{HS}}}{\sqrt{2n}}
\right)
&\le\mathbb{P}\left(\left\{\|A\Tilde{\Pi}^T v\|_2
\le
\frac{\sqrt{c_1}\|AP_{\mathbf{1}_n^\perp}\|_{\mathrm{HS}}}{\sqrt{2n}}\right\}\cap\mathcal{E}\right)+\mathbb{P}(\mathcal{E}^c)\notag\\
&\le 2\exp\left(-c_2\frac{\|AP_{\mathbf{1}_n^\perp}\|_{\mathrm{HS}}^2}{\|A\|^2}\right) +\exp\left(-\frac{c_{\ref{Weightcaptured}}\gamma\|AP_{\mathbf{1}_{n}^\perp}\|_{\mathrm{HS}}^2}{2\|A\|^2}\right)\notag\\
&\le C_{\ref{permutednormestimate}}
\exp\left(
-c_{\ref{permutednormestimate}}
\frac{\|AP_{\mathbf{1}_n^\perp}\|_{\mathrm{HS}}^2}
{\|A\|^2}
\right),
\end{align}
where $C_{\ref{permutednormestimate}}\ge3$ is a sufficiently large constant and $c_{\ref{permutednormestimate}}=\min\{c_2,c_{\ref{Weightcaptured}}\gamma/2,\sqrt{c_1/2}\}$. 

Consequently, in the case
\[
\frac{\|AP_{\mathbf{1}_n^\perp}\|_{\mathrm{HS}}^2}
{\|A\|^2}\ge r_0(\delta,\rho),
\]
since $\Pi\stackrel{d}{=}\Tilde{\Pi}^T$,
we obtain
\[
\mathbb{P}\left(
\|A\Pi v\|_2
\le
c_{\ref{permutednormestimate}}
\frac{\|AP_{\mathbf{1}_n^\perp}\|_{\mathrm{HS}}}{\sqrt n}
\right)\le C_{\ref{permutednormestimate}}
\exp\left(
-c_{\ref{permutednormestimate}}
\frac{\|AP_{\mathbf{1}_n^\perp}\|_{\mathrm{HS}}^2}
{\|A\|^2}
\right).
\]
In the remaining case,
\[
\frac{\|AP_{\mathbf{1}_n^\perp}\|_{\mathrm{HS}}^2}
{\|A\|^2}< r_0(\delta,\rho),
\]
if necessary, the same bounds follows after enlarging $C_{\ref{permutednormestimate}}\ge \exp(c_{\ref{permutednormestimate}}r_0)$. This completes the proof.
\end{proof}
\begin{corollary}\label{completepermutednormestimate}
Let $\delta,\rho\in(0,1)$. There exist constants
$C_{\ref{completepermutednormestimate}},c_{\ref{completepermutednormestimate}}>0$
and $n_{\ref{completepermutednormestimate}}\in\mathbb{N}$, depending only on
$\delta$ and $\rho$, such that the following holds. Let $n\ge n_{\ref{completepermutednormestimate}}$, and let $A$ be a deterministic nonzero $n\times n$ matrix. Let $\pi\sim\operatorname{Unif}(S_n)$ with the corresponding permutation
matrix $\Pi$. Then for any fixed $v\in \mathbb{S}^{n-1}\setminus \operatorname{Cons}(\delta,\rho)$,
\[
\mathbb{P}\left(
\|A\Pi v\|_2
\le
c_{\ref{completepermutednormestimate}}
\frac{\|A\|_{\mathrm{HS}}}{\sqrt n}
\right)
\le
C_{\ref{completepermutednormestimate}}
\exp\left(
-c_{\ref{completepermutednormestimate}}\operatorname{sr}(A)
\right).
\]
\end{corollary}
\begin{proof}
    Note that 
    \begin{align*}
\|AP_{\mathbf{1}_{n}^{\perp}}\|_{\mathrm{HS}}^2=\operatorname{tr}\left(AP_{\mathbf{1}_{n}^{\perp}}A^T\right)=\operatorname{tr}\left(AA^T\right)-\operatorname{tr}\left(A\frac1n \mathbf 1_n\mathbf 1_n^TA^T\right)&=\|A\|_{\mathrm{HS}}^2-\left\|A\frac{1}{\sqrt{n}}\mathbf{1}_n\right\|_2^2\\&\ge\|A\|_{\mathrm{HS}}^2-\|A\|^2.
\end{align*}
If $\operatorname{sr}(A)\ge2$, we have
\[
\frac{\|AP_{\mathbf{1}_{n}^{\perp}}\|_{\mathrm{HS}}^2}{\|A\|^2}\ge\frac{\|A\|_{\mathrm{HS}}^2-\|A\|^2}{\|A\|^2}= \operatorname{sr}(A)-1\ge\frac{1}{2}\operatorname{sr}(A),
\]
and $\|AP_{\mathbf{1}_{n}^{\perp}}\|_{\mathrm{HS}}^2\ge1/2\|A\|^2\operatorname{sr}(A)=1/2\|A\|_{\mathrm{HS}}^2$. By Lemma \ref{permutednormestimate}, we obtain
\[
\mathbb{P}\left(
\|A\Pi v\|_2
\le
c_{\ref{permutednormestimate}}
\frac{\sqrt{2}\|A\|_{\mathrm{HS}}}{2\sqrt n}
\right)
\le
C_{\ref{permutednormestimate}}
\exp\left(
-c_{\ref{permutednormestimate}}
\frac{1}
{2}\operatorname{sr}(A)
\right).
\]
If $\operatorname{sr}(A)<2$, then, since $r_0>1$ is a sufficiently large constant number, $C_{\ref{permutednormestimate}}\ge\exp(c_{\ref{permutednormestimate}}r_0)\ge\exp(c_{\ref{permutednormestimate}})$, the desired probability bound is trivial. Let $C_{\ref{completepermutednormestimate}}=C_{\ref{permutednormestimate}},c_{\ref{completepermutednormestimate}}=c_{\ref{permutednormestimate}}/2,n_{\ref{completepermutednormestimate}}=n_{\ref{permutednormestimate}}$. This completes the proof.
\end{proof}

\subsection{Proof of Theorem 1.3}
\begin{proof}
First, we assume $n\geq n_0(p)$, sufficiently large. When $n<n_0$, the case can be absorbed by enlarging the constants.

Following Proposition \ref{constant estimate}, there exist $C_{\ref{constant estimate}},c_{\ref{constant estimate}},\delta_{\ref{constant estimate}},\rho_{\ref{constant estimate}}$ such that, when $n$ is sufficiently large,
\begin{align}\label{E1probabilitybound}
    \mathbb{P}(\mathcal{E}_1^c)\le e^{-c_{\ref{constant estimate}}n},
\end{align}
where $\mathcal{E}_1:=\{\inf_{v\in \operatorname{Cons}(\delta_{\ref{constant estimate}},\rho_{\ref{constant estimate}})}\|{Q_n v}\|_2\ge C_{\ref{constant estimate}}\sqrt{pn}\}$.

Choose a measurable function $u=u(Q_n)\in\mathbb S^{n-1}$ such that $\|Q_nu\|_2=s_n(Q_n)$. Let $\pi\sim\operatorname{Unif}(S_n)$ be independent of $Q_n$, and let $\Pi$ denote its permutation matrix. Then $u$ and $\Pi$ are independent. Set $\kappa=C_{\ref{constant estimate}}\sqrt{p}/2$, and let $\varepsilon>0$.

Let
\[
\mathcal{E}_2:=\left\{s_{n}(Q_n)\le\min\left\{\frac{c_{\ref{completepermutednormestimate}}}{\varepsilon^2\sqrt{n}},\kappa\sqrt{n}\right\}\right\}.
\]
Applying the upper bound in \cite[Theorem 1.1]{lianupper}, we obtain
\begin{align}\label{E2probabilitybound}
\mathbb{P}(\mathcal{E}_2^c)\le C_1\left(\Tilde{t} +\frac{1}{\sqrt{d}}\right)\le C_1\left(\frac{\varepsilon p^{1/4}}{\sqrt{c_{\ref{completepermutednormestimate}}}}+\frac{p^{1/4}}{\sqrt{\kappa n}}+\frac{1}{\sqrt{d}}\right)\le C_1'\left(\varepsilon+\frac{1}{\sqrt{d}}\right),
\end{align}

where $\Tilde{t}=\max\{\varepsilon p^{1/4}/\sqrt{c_{\ref{completepermutednormestimate}}},p^{1/4}/\sqrt{\kappa n}\}$ and $C_1'>0$ depends only on $p$. 

On $\mathcal{E}_1\cap\mathcal{E}_2$, since
\[
\|Q_nu\|_2=s_n(Q_n)\le \kappa\sqrt{n}< C_{\ref{constant estimate}}\sqrt{pn},
\]
hence, $u\in \mathbb{S}^{n-1}\setminus\operatorname{Cons}(\delta_{\ref{constant estimate}},\rho_{\ref{constant estimate}})$. 

Let
\[
\mathcal{E}_3^c:=\left\{\left\{\|M^{-1}\Pi u\|_2\le c_{\ref{completepermutednormestimate}}\frac{\|M^{-1}\|_{\mathrm{HS}}}{\sqrt{n}}\right\}\cap\{u\in \mathbb{S}^{n-1}\setminus\operatorname{Cons}(\delta_{\ref{constant estimate}},\rho_{\ref{constant estimate}})\}\right\},
\]
by Corollary \ref{completepermutednormestimate} and the independence of $u$ and $\Pi$,
\begin{align}\label{E3probabilitybound}
    \mathbb{P}(\mathcal{E}_3^c)&=\mathbb{E}\left(\mathbf{1}_{\{u\notin\operatorname{Cons}(\delta_{\ref{constant estimate}},\rho_{\ref{constant estimate}})\}}\mathbb{P}\left(\|M^{-1}\Pi u\|_2\le c_{\ref{completepermutednormestimate}}\frac{\|M^{-1}\|_{\mathrm{HS}}}{\sqrt{n}}\,\middle|\,u\right)\right)\notag\\&\le C_{\ref{completepermutednormestimate}}
\exp\left(
-c_{\ref{completepermutednormestimate}}\operatorname{sr}(M^{-1})\right).
\end{align}
On $\mathcal{E}_1\cap\mathcal{E}_{2}\cap\mathcal{E}_3$, we have
\[
s_n(Q_n\Pi^TM)\le\frac{s_n(Q_n)}{\|M^{-1}\Pi u\|_2}\le\frac{\min\left\{c_{\ref{completepermutednormestimate}}/(\varepsilon^2\sqrt{n}),\kappa\sqrt{n}\right\}}{c_{\ref{completepermutednormestimate}}\|M^{-1}\|_{\mathrm{HS}}/\sqrt{n}}\le\frac{1}{\varepsilon^2\|M^{-1}\|_{\mathrm{HS}}}.
\]
Combining (\ref{E1probabilitybound}), (\ref{E2probabilitybound}), (\ref{E3probabilitybound}), we obtain
\begin{align} \mathbb{P}\left(s_n(Q_nM)>\frac{1}{\varepsilon^2\|M^{-1}\|_{\mathrm{HS}}}\right)&=\mathbb{P}\left(s_n(Q_n\Pi^TM)>\frac{1}{\varepsilon^2\|M^{-1}\|_{\mathrm{HS}}}\right)\notag\\
&\le\mathbb{P}(\mathcal{E}_1^c\cup\mathcal{E}_2^c\cup\mathcal{E}_3^c)\notag\\
&\le e^{-c_{\ref{constant estimate}}n}+C_1'\left(\varepsilon+\frac{1}{\sqrt{d}}\right)+C_{\ref{completepermutednormestimate}}
\exp\left(
-c_{\ref{completepermutednormestimate}}\operatorname{sr}(M^{-1})\right)\notag\\
&\le C_{\ref{theorem 1.2}} \left(\varepsilon +\frac{1}{\sqrt{d}}+e^{-c_{\ref{theorem 1.2}} \operatorname{sr}(M^{-1})}\right),\notag
\end{align}
where we choose $C_{\ref{theorem 1.2}},c_{\ref{theorem 1.2}}>0$ depending only on $p$, and the term $e^{-c_{\ref{constant estimate}}n}$ is absorbed into the $1/\sqrt{d}$ since $d=pn$. This completes the proof.
\end{proof}

\section*{Declaration of AI use}
During the preparation of this manuscript, the authors used OpenAI's ChatGPT and Codex to assist with language editing, proofreading, and the organization and presentation of the manuscript. All mathematical ideas, results, proofs, and conclusions are those of the authors. The authors independently checked all proofs and take full responsibility for the mathematical claims and the final manuscript.

\printbibliography

@article {numericalinverting,
    AUTHOR = {von Neumann, John and Goldstine, H. H.},
     TITLE = {Numerical inverting of matrices of high order},
   JOURNAL = {Bull. Amer. Math. Soc.},
    VOLUME = {53},
      YEAR = {1947},
     PAGES = {1021--1099},
       DOI = {10.1090/S0002-9904-1947-08909-6},
       URL = {https://doi.org/10.1090/S0002-9904-1947-08909-6},
}

@book {collectedworks,
   author    = {von Neumann, John},
  title     = {Collected works. Vol. V: Design of computers, theory of automata and numerical analysis},
  publisher = {The Macmillan Company},
  address   = {New York},
  year      = {1963}
}

@article {eigenvalues,
    AUTHOR = {Edelman, Alan},
     TITLE = {Eigenvalues and condition numbers of random matrices},
   JOURNAL = {SIAM J. Matrix Anal. Appl.},
    VOLUME = {9},
      YEAR = {1988},
    NUMBER = {4},
     PAGES = {543--560},
       DOI = {10.1137/0609045},
       URL = {https://doi.org/10.1137/0609045},
}

@article {conditionnumbers,
    AUTHOR = {Szarek, Stanis{\l}aw J.},
     TITLE = {Condition numbers of random matrices},
   JOURNAL = {J. Complexity},
    VOLUME = {7},
      YEAR = {1991},
    NUMBER = {2},
     PAGES = {131--149},
       DOI = {10.1016/0885-064X(91)90002-F},
       URL = {https://doi.org/10.1016/0885-064X(91)90002-F},
}

@article {thelittle,
    AUTHOR = {Rudelson, Mark and Vershynin, Roman},
     TITLE = {The {L}ittlewood-{O}fford problem and invertibility of random
              matrices},
   JOURNAL = {Adv. Math.},
    VOLUME = {218},
      YEAR = {2008},
    NUMBER = {2},
     PAGES = {600--633},
       DOI = {10.1016/j.aim.2008.01.010},
       URL = {https://doi.org/10.1016/j.aim.2008.01.010},
}

@article {coveringof,
    AUTHOR = {Rebrova, Elizaveta and Tikhomirov, Konstantin},
     TITLE = {Coverings of random ellipsoids, and invertibility of matrices
              with i.i.d. heavy-tailed entries},
   JOURNAL = {Israel J. Math.},
    VOLUME = {227},
      YEAR = {2018},
    NUMBER = {2},
     PAGES = {507--544},
       DOI = {10.1007/s11856-018-1732-y},
       URL = {https://doi.org/10.1007/s11856-018-1732-y},
}

@article {jaindiscrete,
    AUTHOR = {Jain, Vishesh and Sah, Ashwin and Sawhney, Mehtaab},
     TITLE = {Singularity of discrete random matrices},
   JOURNAL = {Geom. Funct. Anal.},
    VOLUME = {31},
      YEAR = {2021},
    NUMBER = {5},
     PAGES = {1160--1218},
       DOI = {10.1007/s00039-021-00580-6},
       URL = {https://doi.org/10.1007/s00039-021-00580-6},
}

@article {symmetricsubgaussian,
    AUTHOR = {Campos, Marcelo and Jenssen, Matthew and Michelen, Marcus and
              Sahasrabudhe, Julian},
     TITLE = {The least singular value of a random symmetric matrix},
   JOURNAL = {Forum Math. Pi},
    VOLUME = {12},
      YEAR = {2024},
     PAGES = {Paper No. e3, 69},
       DOI = {10.1017/fmp.2023.29},
       URL = {https://doi.org/10.1017/fmp.2023.29},
}

@article {cookregular,
    AUTHOR = {Cook, Nicholas A.},
     TITLE = {On the singularity of adjacency matrices for random regular
              digraphs},
   JOURNAL = {Probab. Theory Related Fields},
    VOLUME = {167},
      YEAR = {2017},
    NUMBER = {1-2},
     PAGES = {143--200},
       DOI = {10.1007/s00440-015-0679-8},
       URL = {https://doi.org/10.1007/s00440-015-0679-8},
}

@article {jainregular,
    AUTHOR = {Jain, Vishesh and Sah, Ashwin and Sawhney, Mehtaab},
     TITLE = {The smallest singular value of dense random regular digraphs},
   JOURNAL = {Int. Math. Res. Not. IMRN},
      YEAR = {2022},
    NUMBER = {24},
     PAGES = {19300--19334},
       DOI = {10.1093/imrn/rnab247},
       URL = {https://doi.org/10.1093/imrn/rnab247},
}

@article {jainapprox,
    AUTHOR = {Jain, Vishesh},
     TITLE = {Approximate {S}pielman--{T}eng theorems for the least singular value of random combinatorial matrices},
   JOURNAL = {Israel J. Math.},
    VOLUME = {242},
      YEAR = {2021},
    NUMBER = {1},
     PAGES = {461--500},
       DOI = {10.1007/s11856-021-2144-y},
       URL = {https://doi.org/10.1007/s11856-021-2144-y},
}

@article {litvakdigraphs,
    AUTHOR = {Litvak, Alexander E. and Lytova, Anna and Tikhomirov, Konstantin E. and Tomczak-Jaegermann, Nicole and Youssef, Pierre},
     TITLE = {Adjacency matrices of random digraphs: singularity and anti-concentration},
   JOURNAL = {J. Math. Anal. Appl.},
    VOLUME = {445},
      YEAR = {2017},
    NUMBER = {2},
     PAGES = {1447--1491},
       DOI = {10.1016/j.jmaa.2016.08.020},
       URL = {https://doi.org/10.1016/j.jmaa.2016.08.020},
}

@article {litvakshiftedregular,
    AUTHOR = {Litvak, Alexander E. and Lytova, Anna and Tikhomirov, Konstantin E. and Tomczak-Jaegermann, Nicole and Youssef, Pierre},
     TITLE = {The smallest singular value of a shifted {$d$}-regular random square matrix},
   JOURNAL = {Probab. Theory Related Fields},
    VOLUME = {173},
      YEAR = {2019},
    NUMBER = {3-4},
     PAGES = {1301--1347},
       DOI = {10.1007/s00440-018-0852-y},
       URL = {https://doi.org/10.1007/s00440-018-0852-y},
}

@article {nguyencombinatorial,
    AUTHOR = {Nguyen, Hoi H.},
     TITLE = {On the singularity of random combinatorial matrices},
   JOURNAL = {SIAM J. Discrete Math.},
    VOLUME = {27},
      YEAR = {2013},
    NUMBER = {1},
     PAGES = {447--458},
       DOI = {10.1137/110853157},
       URL = {https://doi.org/10.1137/110853157},
}

@article {lianupper,
    AUTHOR = {Li, Dongbin and Litvak, Alexander E. and Yu, Tingzhou},
     TITLE = {An upper bound on the smallest singular value of dense random
              combinatorial matrices},
   JOURNAL = {J. Complexity},
    VOLUME = {95},
      YEAR = {2026},
     PAGES = {Paper No. 102031, 17},
       DOI = {10.1016/j.jco.2026.102031},
       URL = {https://doi.org/10.1016/j.jco.2026.102031},
}

@article{lianlitvakyucircular,
  author        = {Li, Dongbin and Litvak, Alexander E. and Yu, Tingzhou},
  title         = {Least singular values of shifted sparse random combinatorial matrices},
  journal       = {arXiv preprint arXiv:2604.10446},
  year          = {2026},
  eprint        = {2604.10446},
  archiveprefix = {arXiv},
  primaryclass  = {math.PR},
  doi           = {10.48550/arXiv.2604.10446},
  url           = {https://arxiv.org/abs/2604.10446},
}

@article{yusigned,
  title={The smallest singular value of signed random combinatorial matrices},
  author={Yu, Kexin},
  journal={arXiv preprint arXiv:2604.11761},
  year={2026}
}

@article{litvak2004random,
  title   = {Random {E}uclidean embeddings in spaces of bounded volume ratio},
  author  = {Litvak, Alexander E. and Pajor, Alain and Rudelson, Mark and Tomczak-Jaegermann, Nicole and Vershynin, Roman},
  journal = {C. R. Math. Acad. Sci. Paris},
  volume  = {339},
  number  = {1},
  pages   = {33--38},
  year    = {2004}
}

@article{livshyts2021smallestTK,
  title   = {The smallest singular value of inhomogeneous square random matrices},
  author  = {Livshyts, Galyna V. and Tikhomirov, Konstantin and Vershynin, Roman},
  journal = {Ann. Probab.},
  volume  = {49},
  number  = {3},
  pages   = {1286--1309},
  year    = {2021},
  doi     = {10.1214/20-AOP1481}
}

@article{tao2009inverse,
  title   = {Inverse Littlewood-Offord theorems and the condition number of random discrete matrices},
  author  = {Tao, Terence and Vu, Van},
  journal = {Ann. of Math. (2)},
  volume  = {169},
  number  = {2},
  pages   = {595--632},
  year    = {2009},
  doi     = {10.4007/annals.2009.169.595}
}

@article{tao2010random,
  title   = {Random matrices: The distribution of the smallest singular values},
  author  = {Tao, Terence and Vu, Van},
  journal = {Geom. Funct. Anal.},
  volume  = {20},
  number  = {1},
  pages   = {260--297},
  year    = {2010},
  doi     = {10.1007/s00039-010-0057-8}
}

@article{rudelson2008invertibility,
  title={Invertibility of random matrices: norm of the inverse},
  author={Rudelson, Mark},
  journal={Ann. of Math. (2)},
  volume={168},
  number={2},
  pages={575--600},
  year={2008},
  doi={10.4007/annals.2008.168.575}
}

@article{rudelson2009smallest,
  title   = {Smallest singular value of a random rectangular matrix},
  author  = {Rudelson, Mark and Vershynin, Roman},
  journal = {Comm. Pure Appl. Math.},
  volume  = {62},
  number  = {12},
  pages   = {1707--1739},
  year    = {2009},
  doi     = {10.1002/cpa.20294}
}

@article{letwin2026smallest,
  title={On the smallest singular value of the product of random and deterministic matrices},
  author={Letwin, Brayden and Polavarapu, Achintya Raya},
  journal={arXiv preprint arXiv:2607.06785},
  year={2026}
}

@book{highdimensional,
 author = {Vershynin, Roman},
 title = {High-dimensional probability. {An} introduction with applications in data science},
 edition = {2nd edition},
 fseries = {Cambridge Series in Statistical and Probabilistic Mathematics},
 series = {Camb. Ser. Stat. Probab. Math.},
 volume = {58},
 isbn = {978-1-00-949064-1; 978-1-00-949067-2},
 year = {2026},
 publisher = {Cambridge: Cambridge University Press},
 language = {English},
 doi = {10.1017/9781009490672},
 zbMATH = {8141018},
 Zbl = {1579.60005}
}

@article{sah2025spielmanteng,
  author  = {Sah, Ashwin and Sahasrabudhe, Julian and Sawhney, Mehtaab},
  title   = {On the {Spielman--Teng} Conjecture},
  journal = {Geom. Funct. Anal.},
  volume  = {35},
  number  = {2},
  pages   = {633--671},
  year    = {2025},
  doi     = {10.1007/s00039-025-00707-z},
  url     = {https://doi.org/10.1007/s00039-025-00707-z}
}

@article{ferberjainluhsamotij,
  author  = {Ferber, Asaf and Jain, Vishesh and Luh, Kyle and Samotij, Wojciech},
  title   = {On the counting problem in inverse {Littlewood--Offord} theory},
  journal = {J. Lond. Math. Soc. (2)},
  volume  = {103},
  number  = {4},
  pages   = {1333--1362},
  year    = {2021},
  doi     = {10.1112/jlms.12409}
}

@article{shi2026signedupper,
  author  = {Shi, Yuxiao and Song, Zeyan and Yu, Kexin and Yu, Tingzhou},
  title   = {An upper bound on the smallest singular value of dense signed random combinatorial matrices},
  journal = {arXiv preprint arXiv:2610.33427},
  year    = {2026},
  url     = {https://arxiv.org/abs/2610.33427}
}

@article{vershynin2011products,
  author  = {Vershynin, Roman},
  title   = {Spectral norm of products of random and deterministic matrices},
  journal = {Probab. Theory Related Fields},
  volume  = {150},
  number  = {3-4},
  pages   = {471--509},
  year    = {2011},
  doi     = {10.1007/s00440-010-0281-z}
}

@article{xiyangyin2017product,
  author  = {Xi, Haokai and Yang, Fan and Yin, Jun},
  title   = {Local circular law for the product of a deterministic matrix with a random matrix},
  journal = {Electron. J. Probab.},
  volume  = {22},
  eid     = {60},
  pages   = {1--77},
  year    = {2017},
  doi     = {10.1214/17-EJP76}
}

@article{kopel2020products,
  author  = {Kopel, Phil and O'Rourke, Sean and Vu, Van},
  title   = {Random matrix products: universality and least singular values},
  journal = {Ann. Probab.},
  volume  = {48},
  number  = {3},
  pages   = {1372--1410},
  year    = {2020},
  doi     = {10.1214/19-AOP1396}
}

@article{chaudhuri2021products,
  author  = {Chaudhuri, Rohit and Jain, Vishesh and Pillai, Natesh S.},
  title   = {Universality and least singular values of random matrix products: a simplified approach},
  journal = {Bernoulli},
  volume  = {27},
  number  = {4},
  pages   = {2519--2531},
  year    = {2021},
  doi     = {10.3150/20-BEJ1320}
}

@article{mourrat2026products,
  author  = {Mourrat, Jean-Christophe},
  title   = {Operator {$\ell^\infty\to\ell^\infty$} norm of products of random matrices with iid entries},
  journal = {Electron. J. Probab.},
  volume  = {31},
  eid     = {29},
  pages   = {1--26},
  year    = {2026},
  doi     = {10.1214/26-EJP1488}
}

@inproceedings{li2020fedavg,
  author    = {Li, Xiang and Huang, Kaixuan and Yang, Wenhao
               and Wang, Shusen and Zhang, Zhihua},
  title     = {On the Convergence of {FedAvg} on {Non-IID} Data},
  booktitle = {International Conference on Learning Representations},
  year      = {2020},
  url       = {https://openreview.net/forum?id=HJxNAnVtDS},
  eprint    = {1907.02189},
  eprinttype = {arxiv}
}

@inproceedings{fraboni2021clustered,
  author    = {Fraboni, Yann and Vidal, Richard and Kameni, Laetitia
               and Lorenzi, Marco},
  title     = {Clustered Sampling: Low-Variance and Improved
               Representativity for Clients Selection in Federated Learning},
  booktitle = {Proceedings of the 38th International Conference
               on Machine Learning},
  series    = {Proceedings of Machine Learning Research},
  volume    = {139},
  pages     = {3407--3416},
  publisher = {PMLR},
  year      = {2021},
  url       = {https://proceedings.mlr.press/v139/fraboni21a.html}
}

@article{so2023securing,
  author       = {So, Jinhyun and Ali, Ramy E. and G{\"u}ler, Ba\c{s}ak
                  and Jiao, Jiantao and Avestimehr, A. Salman},
  title        = {Securing Secure Aggregation: Mitigating Multi-Round
                  Privacy Leakage in Federated Learning},
  journaltitle = {Proceedings of the AAAI Conference on Artificial Intelligence},
  volume       = {37},
  number       = {8},
  pages        = {9864--9873},
  year         = {2023},
  doi          = {10.1609/aaai.v37i8.26177},
  eprint       = {2106.03328},
  eprinttype   = {arxiv},
  note         = {Extended version available as arXiv:2106.03328}
}

@inproceedings{lam2021gradient,
  author    = {Lam, Maximilian and Wei, Gu-Yeon and Brooks, David
               and Reddi, Vijay Janapa and Mitzenmacher, Michael},
  title     = {Gradient Disaggregation: Breaking Privacy in Federated
               Learning by Reconstructing the User Participant Matrix},
  booktitle = {Proceedings of the 38th International Conference
               on Machine Learning},
  series    = {Proceedings of Machine Learning Research},
  volume    = {139},
  pages     = {5959--5968},
  publisher = {PMLR},
  year      = {2021},
  eprint    = {2106.06089},
  eprinttype = {arxiv}
}

@article{tran,
  author  = {Tran, Tuan},
  title   = {The smallest singular value of random combinatorial matrices},
  journal = {arXiv preprint arXiv:2007.06318},
  year    = {2020}
}

@article{litvak2005smallest,
  author  = {Litvak, Alexander E. and Pajor, Alain and Rudelson, Mark
             and Tomczak-Jaegermann, Nicole},
  title   = {Smallest singular value of random matrices and geometry
             of random polytopes},
  journal = {Adv. Math.},
  volume  = {195},
  number  = {2},
  pages   = {491--523},
  year    = {2005},
  doi     = {10.1016/j.aim.2004.08.004}
}

@article{fernandez2025logconcave,
  author  = {Fernandez V, Manuel and Livshyts, Galyna V. and Mui, Stephanie},
  title   = {On the Smallest Singular Value of Log-Concave Random Matrices},
  journal = {arXiv preprint arXiv:2508.17745},
  year    = {2025}
}

@article{dabagia2024inhomogeneous,
  author  = {Dabagia, Max and Fernandez V, Manuel},
  title   = {The smallest singular value of inhomogenous random rectangular matrices},
  journal = {arXiv preprint arXiv:2408.14389v2},
  year    = {2024}
}

@article{fernandez2024distance,
  author  = {Fernandez V, Manuel},
  title   = {A distance theorem for inhomogenous random rectangular matrices},
  journal = {arXiv preprint arXiv:2408.06309v2},
  year    = {2024}
}

@inproceedings{tikhomirov2022quantitative,
  author    = {Tikhomirov, Konstantin},
  title     = {Quantitative invertibility of non-Hermitian random matrices},
  booktitle = {Proceedings of the International Congress of Mathematicians 2022},
  volume    = {4},
  pages     = {3292--3313},
  publisher = {EMS Press},
  address   = {Berlin},
  year      = {2023},
  doi       = {10.4171/ICM2022/10}
}

@article{litvak2022singularity,
  author  = {Litvak, Alexander E. and Tikhomirov, Konstantin E.},
  title   = {Singularity of sparse {Bernoulli} matrices},
  journal = {Duke Math. J.},
  volume  = {171},
  number  = {5},
  pages   = {1135--1233},
  year    = {2022},
  doi     = {10.1215/00127094-2021-0056}
}

@article{yu2025sparsediscrete,
  author  = {Yu, Kexin},
  title   = {The smallest singular value of sparse discrete random matrices},
  journal = {arXiv preprint arXiv:2507.17741},
  year    = {2025}
}

@article{basakrudelson2017invertibility,
  author  = {Basak, Anirban and Rudelson, Mark},
  title   = {Invertibility of sparse non-{Hermitian} matrices},
  journal = {Adv. Math.},
  volume  = {310},
  pages   = {426--483},
  year    = {2017},
  doi     = {10.1016/j.aim.2017.02.009}
}

@article{basak2021sharp,
  author  = {Basak, Anirban and Rudelson, Mark},
  title   = {Sharp transition of the invertibility of the adjacency matrices
             of sparse random graphs},
  journal = {Probab. Theory Related Fields},
  volume  = {180},
  number  = {1--2},
  pages   = {233--308},
  year    = {2021},
  doi     = {10.1007/s00440-021-01038-4}
}

@article{arioli2007partial,
  author  = {Arioli, Mario and Baboulin, Marc and Gratton, Serge},
  title   = {A Partial Condition Number for Linear Least Squares Problems},
  journal = {SIAM Journal on Matrix Analysis and Applications},
  volume  = {29},
  number  = {2},
  pages   = {413--433},
  year    = {2007},
  doi     = {10.1137/050643088},
  url     = {https://doi.org/10.1137/050643088}
}

@article{baboulin2009conditioning,
  author  = {Baboulin, Marc and Dongarra, Jack and Gratton, Serge and Langou, Julien},
  title   = {Computing the Conditioning of the Components of a Linear Least-Squares Solution},
  journal = {Numerical Linear Algebra with Applications},
  volume  = {16},
  number  = {7},
  pages   = {517--533},
  year    = {2009},
  doi     = {10.1002/nla.627},
  url     = {https://doi.org/10.1002/nla.627}
}

@article{kwan2019anticoncentration,
  author  = {Kwan, Matthew and Sudakov, Benny and Tran, Tuan},
  title   = {Anticoncentration for subgraph statistics},
  journal = {J. Lond. Math. Soc.},
  volume  = {99},
  number  = {3},
  pages   = {757--777},
  year    = {2019},
  doi     = {10.1112/jlms.12192},
  url     = {https://doi.org/10.1112/jlms.12192}
}

\vspace{3cm}
\par\noindent
\textsc{Yuxiao Shi}
\\
\textit{Email address:} \texttt{yshi6821@gmail.com}

\par\noindent
Zhongtai Securities Institute for Financial Studies, Shandong University, Jinan, 250100, China.

\par\noindent
\textsc{Kexin Yu}
\\
\textit{Email address:} \texttt{Kexinyu01@gmail.com}

\par\noindent
School of Mathematics, Shandong University, Jinan, 250100, China.

\par\noindent
\textsc{Tingzhou Yu} 
\\
\textit{Email address:} \texttt{tingzho1@ualberta.ca}

\par\noindent
Department of Mathematical and Statistical Sciences, University of Alberta, Edmonton, AB, T6G 2R3, Canada.

\end{document}